\documentclass[11pt,a4paper,reqno]{amsart}

\usepackage{times} % assumes new font selection scheme installed
\usepackage{amsmath,dsfont} % assumes amsmath package installed
\usepackage{amssymb}  % assumes amsmath package installed
\usepackage{amsthm}
\usepackage{latexsym}
\usepackage{amsfonts,bbm}
\usepackage{xcolor}
\usepackage{mathtools}
\usepackage{enumerate}
\usepackage{cite}
\usepackage[foot]{amsaddr}
\usepackage{algorithm}%
\usepackage{algorithmicx}%
\usepackage{algpseudocode}%
\usepackage[
  a4paper,
  margin=3cm
]{geometry}
\usepackage{multicol}
\usepackage{enumitem}

\usepackage{tikz}
\usetikzlibrary{calc,positioning, arrows.meta, shapes.geometric}
\usepackage{xargs}
\usepackage[hidelinks]{hyperref}
\usepackage{textcomp}

\usepackage{tikz}
\usetikzlibrary{positioning, arrows.meta, shapes.geometric}
\usepackage{graphicx}
\usepackage{xcolor}

\usepackage[colorinlistoftodos,prependcaption,textsize=small]{todonotes}
\usepackage{xargs}  
\usepackage{caption}
\usepackage[labelformat=simple]{subcaption}

\newcommand{\stack}[2]{\left[ \begin{smallmatrix} #1 \\ #2 \end{smallmatrix} \right]}

\newcommand{\roundstack}[2]{\left( \begin{smallmatrix} #1 \\ #2 \end{smallmatrix} \right)}

\newtheorem{theorem}{Theorem}[section]

\newtheorem{lemma}[theorem]{Lemma}

\newtheorem{proposition}[theorem]{Proposition}
\newtheorem{definition}{Definition}

\numberwithin{equation}{section}
\allowdisplaybreaks

\newcommand{\R}{\mathbb{R}}
\newcommand{\N}{\mathbb{N}}
\newcommand{\calA}{\mathcal A}

\newcommand{\calC}{\mathcal C}

\newcommand{\calM}{\mathcal M}

\newcommand{\calS}{\mathcal S}

\newcommand{\dom}[1]{D(#1)}

\newcommand{\ran}[1]{R(#1)}
\newcommand{\zer}[1]{\operatorname{zer}(#1)}

\DeclareMathOperator{\prox}{prox}

{\left(\begin{smallmatrix}}%            begin code
{\end{smallmatrix}\right)}%             end code

{\left[\begin{smallmatrix}}%            begin code
{\end{smallmatrix}\right]}%             end code

\title[Suboptimal model predictive control as an operator splitting scheme]{Coupling optimization algorithms and monotone control systems: Suboptimal model predictive control as an operator splitting scheme}

\author{Till Preuster$^1$}\address{$^1$Junior Professorship Numerical Mathematics, Faculty of Mathematics, Chemnitz University of Technology, Germany\\ Mail: \textsc{\{till.preuster,manuel.schaller\}@math.tu-chemnitz.de}}
\author{Hannes Gernandt$^2$}\address{$^2$School of Mathematics and Natural Sciences, University of Wuppertal, Germany\\ Mail: \textsc{gernandt@uni-wuppertal.de}}
\author{Manuel Schaller$^1$}

\thanks{This work was funded by the Deutsche Forschungsgemeinschaft (DFG, German Research Foundation) – Project-ID 531152215 – CRC 1701.}

\begin{document}

\begin{abstract}
We propose a framework for suboptimal model predictive control (MPC) based on the interconnection of monotone dynamical systems, such as port-Hamiltonian systems. In contrast to classical MPC formulations, where the optimizer is treated as an instantaneous mapping, we model both the plant and the optimizer as dynamical systems and couple them through a structured interconnection. This leads to a continuous-time closed-loop formulation governed by (quasi-)monotone operators. Within this setting, we establish well-posedness of the coupled optimizer-plant dynamics and provide a unified interpretation of suboptimal MPC schemes. In particular, we reveal a direct connection between iterative optimization algorithms and dynamical control systems theory by showing that standard suboptimal MPC algorithms can be understood as time discretizations of the underlying continuous-time dynamics via operator splitting methods.
\end{abstract}

\maketitle
\smallskip
\section{Introduction}\label{sec1}

Model Predictive Control (MPC; \cite{Gruene17,RawlMayn17}) is a well-established feedback control scheme for nonlinear dynamical control systems. One reason for its success in a wide range of applications is its ability to handle nonlinear control systems and explicitly incorporate state and control constraints. In MPC, the feedback supplied to the nonlinear system under control is obtained via the solution of an optimal control problem. More precisely, the system behavior is predicted on a finite time horizon and, using a suitably designed cost functions, an optimization problem is solved. Then, a first portion in time of the optimal control is fed back into the system and the process is repeated. In this way, one obtains an optimization-based feedback controller for which a plethora of results regarding constraint satisfaction, stability and economic performance (in view of the chosen cost functional) exist. The fact that the core of the feedback computation consists of solving an optimal control problem leads to a second main aspect of the success of MPC: After designing a suitable optimal control problem offline, the online implementation of the method may be built on highly powerful optimization methods and software, hence providing a bridge to another very mature area of mathematics.

However, even when using highly specialized optimization algorithms to compute the feedback, the requirement of solving optimal control problems in real time poses significant computational challenges, in particular when dealing with applications involving strongly nonlinear models and partial differential equations. This has motivated the development of \emph{suboptimal MPC} methods, where approximate solutions are used as a feedback, while still ensuring stability and feasibility of the closed-loop system. Early results such as \cite{ScokMayn99} and \cite{pannocchia2011inherently} show that feasible suboptimal solutions, combined with appropriate terminal ingredients, are sufficient to guarantee closed-loop stability. Prominent schemes are real-time iteration (RTI) schemes~\cite{DiehBock05}, see also the recent work \cite{Zanelli21} for a viewpoint of interconnected dynamical systems, instant MPC, where only a single optimization step is performed at each sampling instant \cite{feller2013barrier,YoshInou2019}. For port-Hamiltonian (pH) systems, suboptimal MPC can be formulated as a power-preserving interconnection of two pH systems: one describing the plant and one describing the optimizer dynamics~\cite{Pham22}.

In this sense, suboptimal MPC schemes may be interpreted abstractly as the interconnection of two dynamical systems: the optimizer dynamics and the nonlinear plant to be controlled. Classical MPC formulations implicitly assume that the optimizer associated with the underlying optimal control problem can be computed instantaneously. Suboptimal MPC approaches relax this assumption by replacing the exact optimizer with an approximation obtained through a finite number of optimization steps.
While the plant itself is a physical dynamical control system, many optimization algorithms admit a natural interpretation as dynamical systems as well. This viewpoint not only opens up a wide toolkit of system theory when considering autonomous systems, but also may be used to interconnect optimization algorithms with other dynamical control systems. This interpretation of optimization methods has a long history. Early examples include Polyak's heavy-ball method \cite{polyak1964some} and the accelerated schemes introduced by Nesterov, whose continuous-time limits have been studied in \cite{muehlebach2019dynamical,su2016differential}. Another seminal contribution is the work of Lions and Mercier \cite{lions1979splitting}, which analyzed operator-splitting methods such as the Peaceman--Rachford scheme \cite{peaceman1955numerical}. We further refer to the monographs \cite{helmke2012optimization,miller2026structureanalysissynthesisfirstorder} for a broader systems-theoretic perspective on optimization dynamics.
In recent years, dynamical formulations of optimization algorithms have attracted renewed attention, particularly in the context of inertial and accelerated dynamics \cite{attouch2018fast,attouch2014dynamical,boct2015inertial}. Moreover, primal--dual optimization dynamics have been studied from a control-theoretic perspective, including asymptotic convergence properties \cite{cherukuri2016asymptotic}, exponential stability results \cite{qu2018exponential}, and Lyapunov-based analyses \cite{cherukuri2017role}.

\medskip

\textbf{Contribution}.
Motivated by these developments, this paper proposes a framework for suboptimal predictive control based on the interconnection of monotone systems; one being the optimization algorithm and the other one being the physical system to be controlled. The main idea is to model both the plant and the optimizer as dynamical systems and to couple them through a structured interconnection. We show that the resulting closed-loop system is governed by (quasi-)monotone operators and admits a well-posed dynamical formulation. Moreover, we establish a connection between MPC and suboptimal MPC schemes and operator splitting methods, in particular Lie--Trotter splittings. This perspective provides a unified interpretation of iterative optimization algorithms as time discretizations of the optimizer dynamics.

Overall, the proposed approach provides a bridge between (suboptimal) MPC methods, system theory, and optimization theory, offering new insights into the design and analysis of computationally efficient MPC schemes.

\medskip

\textbf{Outline}. This paper is structured as follows. In Section~\ref{sec:MPC} we define the Model Predictive Control algorithm and its suboptimal variant. Then, in Section~\ref{sec:interconnection} we introduce the system class that is used to both model optimization method and plant under control. Importantly, this method describes an \emph{open system} that allows for a subsequent coupling, leading to a coupled system of differential systems (in case of suboptimal MPC) or differential-algebraic system (in case of MPC), where the algebraic condition corresponds to the optimality condition. Then, in Section~\ref{sec:sampleholdMPC} we show that based on these coupled systems, the classical MPC methods are numerical splitting schemes applied to the coupled optimizer-plant dynamics. Last, in Section~\ref{sec:opt}, we broaden the scope of underlying optimization schemes, showing that proximal-gradient methods, forward-backward algorithms and Peaceman-Rachford algorithms fit into the suggested framework.

\section{Model predictive control schemes}\label{sec:MPC}
In this part, we introduce the stabilizing Model Predictive Control (MPC) schemes to be considered in this work. 
The central goal of stabilizing MPC is to obtain a  stabilizing controller for a given dynamical control system with minimal costs and while satisfying the input or state constraints.
In Subsection~\ref{subsec:mpc_algo}, we present the standard MPC algorithm, while in Subsection~\ref{subsec:subopt_MPC} we put forward its suboptimal counterpart.

Assume that we are given a continuous-time control system 
\begin{equation}\label{eq:control_system}
    \tfrac{\mathrm{d}}{\mathrm{d}t} x_p(t) = f(x_p(t)) + g(u_p(t)), \qquad x_p(0)=x_{p,0} , \qquad t \geq 0,
\end{equation}
with \textit{vector field} $f:\R^n \to \R^n$, \emph{input map} $g: \R^m  \to \R^n$ and \textit{control function} $u_p: [0, \infty) \to \R^m$ and \textit{state trajectory} $x_p: [0, \infty) \to \R^n$. 

A central aim in numerous applications is to stabilize the dynamical system \eqref{eq:control_system} towards a desirable reference state $x_\mathrm{ref}\in \R^n$. More precisely, given a set of admissible control values $U\subset \R^m$ the goal is to find a feedback law $\mu: \R^n \to U$ such that the solution $x(t)$ of the closed-loop system 
\begin{equation}\label{eq:control_system_cl}
    \tfrac{\mathrm{d}}{\mathrm{d}t} x_p(t) = f(x_p(t)) + g(\mu(x_p(t))), \qquad x_p(0)=x_{p,0} , \qquad t \geq 0,
\end{equation}
satisfies 
\begin{equation*}
    \lim_{t \to \infty} x_p(t) = x_\mathrm{ref}.
\end{equation*}
Two major challenges arise in the stabilization of nonlinear control systems. The first challenge is the nonlinear nature of the dynamics. For certain structured classes of systems, several control design methods have been developed to address this issue, including funnel control \cite{berger2021funnel}, passivity-based control \cite{schaft1996l2, ortega2002interconnection}, and approaches based on control Lyapunov functions \cite{sontag1989universal, sontag2013mathematical}. These methods typically rely on the explicit construction of a feedback law that can subsequently be implemented in closed loop.
A second challenge arises from the presence of state and control constraints, which can substantially complicate the design of feedback controllers. In this context, we refer to the recent work \cite{preuster2026optimization}, where it was shown that, for monotone systems such as those considered in this work, these difficulties can be mitigated.

Model Predictive Control (MPC), as introduced below, addresses both challenges, by defining the feedback implicitly through the solution of a finite-horizon optimal control problem that is solved repeatedly online.

\subsection{MPC-algorithm}\label{subsec:mpc_algo}
The basic stabilizing MPC procedure is summarized in  Algorithm~\ref{alg:MPC}. Therein, in each iteration, the following \textit{optimal control problem} is solved
\begin{align}\label{eq:MPC_OCP}
    \begin{aligned}
\min_{(x,u)\in L^2(0,t_f; \R^n) \times L^2(0,t_f; \R^m)} \quad
& \int_0^{t_f} \left(
\tfrac{1}{2} \| x(\tau) - x_\mathrm{ref} \|_{\R^n}^2
+ \tfrac{\alpha}{2} \| u(\tau) - u_\mathrm{ref} \|_{\R^m}^2
\right) \,\mathrm{d}\tau \\
\text{s.t.} \quad
& \dot x(\tau) = \tilde{f}(x(t)) + \tilde{g}(u(t)), \quad x(0) = x_0, \\
& u \in F \subset L^2(0,t_f;\R^m)
\end{aligned}
\end{align}
with $\tilde{f}: \R^n \to \R^n, \tilde{g}: \R^m \to \R^n$ and some convex and closed control constraint set $F \subset L^2(0,t_f;\R^n)$, where $L^2(0,t_f;\R^n)$ is the Lebesgue space of measurable and square integrable functions $f:[0,t_f]\rightarrow\mathbb{R}^n$.
 We set a sampling time $h>0$ and we assume that for each bounded and Lebesgue measurable control input $u:\R\rightarrow\R^m$, there exists 
a unique solution $x(t,x_0,u)$ of \eqref{eq:control_system} with control $u$ at time $t$.
\begin{algorithm}
\caption{Basic MPC algorithm}
\label{alg:MPC}
\begin{algorithmic}[1]
\Statex \textbf{Given}: Initial plant state $x_{p,0}\in \R^n$, sampling time $h>0$.
   \For{each $i = 0,1,2,\dots$}
    \State Solve optimal control problem \eqref{eq:MPC_OCP} with $x_0=x_{p,i}$; denote the obtained optimal control by $u^\star:[0,t_f] \to \R^m$.
    \State Set $u_p(t) \equiv u^\star(0)$ for $t\in [0,h]$.
    \State Set $x_{p,i+1} = x_p(h;x_{p,i},u_p)$ where $x_p$ solves \eqref{eq:control_system}.
\EndFor
\end{algorithmic}
\end{algorithm}
Model Predictive Control is by now a well-established methodology in both control theory and engineering applications. Comprehensive introductions to MPC can be found in the textbooks \cite{Gruene17,RawlMayn17,kouvaritakis2016model}. From an engineering perspective, MPC has proven successful in applications requiring the treatment of constraints and multivariable dynamics; we refer to the survey article \cite{schwenzer2021review}.
On the theoretical side, the mathematical analysis of MPC has been studied extensively over the past decades. Early foundational results on stability and performance guarantees for constrained MPC were established in \cite{mayne2000constrained}. Since then, the framework has been extended in several directions, including dynamic operation \cite{kohler2024analysis}, path-following problems \cite{faulwasser2015nonlinear}, and economic MPC formulations \cite{grune2013economic,angeli2011average}. Moreover, MPC schemes for homogeneous control systems, which arise in robotics applications, have been investigated in \cite{worthmann2015model,coron2020model}. An overview of developments in economic MPC is provided in \cite{faulwasser2018economic}.
More recently, considerable attention has been devoted to data-driven and learning-based MPC approaches, where system models are partially or fully inferred from data. We refer to the overview articles \cite{berberich2025overview,hewing2020learning,strasser2026overview} as well as the recent contribution \cite{bold2024data}.

\subsection{Suboptimal MPC}\label{subsec:suboptmpc_intro}
In practice, computing the optimal solution of the optimal control problem in Step~3 of Algorithm~\ref{alg:MPC} within the available time is often challenging, making it difficult to neglect the resulting suboptimality. This motivates the use of suboptimal solutions. A common approach is to employ an iterative solver for the optimal control problem \eqref{eq:MPC_OCP} and terminate the iteration prematurely. In this way, the basic MPC algorithm is modified as follows.
\begin{algorithm}
\caption{Suboptimal MPC algorithm.}
\label{alg:subopt_MPC}
\begin{algorithmic}[1]
\Statex \textbf{Given}: Initial plant state $x_{p,0}\in \R^n$, sampling time $h>0$, truncation index $j\in \N$.
    \For{each $i = 0,1,2,\dots$}
    \State Approximate optimal control of \eqref{eq:MPC_OCP} with $x_0=x_{p,i}$ by an iterative optimization algorithm; denote the obtained sequence by $(u^l)_{l \in \N}$,  $u^l:[0,t_f] \to \R^m$.
    \State Terminate after $j \in \mathbb{N}$ iterations and set $u_p(t) \equiv u^j(0)$ for $t\in [0,h]$.
    \State Set $x_{p,i+1} = x_p(h;x_{p,i},u_p)$ where $x_p$ solves \eqref{eq:control_system}.
    \EndFor
\end{algorithmic}
\end{algorithm}
A fundamental observation in suboptimal MPC is that optimality can be relaxed while preserving stability. Early results such as \cite{ScokMayn99} and \cite{pannocchia2011inherently} show that feasible suboptimal solutions, combined with appropriate terminal ingredients, are sufficient to guarantee closed-loop stability.

A major line of research focuses on reducing computational complexity. In particular, real-time iteration (RTI) schemes~\cite{DiehBock05} compute the control input by performing a single Newton--SQP step at each sampling instant. These approaches can also be interpreted within continuation methods, see \cite[Section 8.9.2]{RawlMayn17}, or initial-value embedding techniques, see \cite[p.~325ff]{Gruene17}. A common assumption in many of these works is the absence of inequality constraints, see \cite{diehl2007stabilizing}.

A significant extension is provided in \cite{Zanelli21}, where a Lyapunov-based analysis of the coupled optimizer-plant dynamics is developed. This work establishes asymptotic stability in the presence of inequality constraints for sufficiently small sampling times. The framework relies on structural assumptions on both the plant and the optimizer, including Lyapunov-certified stabilizability, Lipschitz continuity properties, and a strict contraction of the optimizer dynamics with respect to the optimal solution.

Further developments include approaches that combine explicit MPC with online optimization~\cite{zeilinger2011real}, as well as schemes ensuring stability and performance improvement without terminal constraints~\cite{graichen2010stability}. Comprehensive discussions of suboptimal (or nonoptimal) MPC can be found in \cite[Section 2.7]{RawlMayn17} and \cite[Section 10.6]{Gruene17}.

Extensions to infinite-dimensional systems have also been investigated. In particular, \cite[Chapter 4]{schrot2025efficient} considers suboptimal MPC for a class of parabolic PDEs, building on the framework of \cite{Zanelli21}. However, additional challenges arise in this setting, such as the lack of compactness of sublevel sets and the difficulty of constructing suitable Lyapunov functionals.

Besides stability, suboptimality of these methods is an active topic of research, see \cite{kara2025} for a finite-time transient suboptimality analysis.

A particular subclass of suboptimal MPC is \emph{instant model predictive control}, where only a single optimization step is performed at each sampling instant. A first contribution in this direction is \cite{feller2013barrier}, which proposes a continuous-time optimization algorithm for linear systems. In \cite{YoshInou2019}, this approach is further analyzed for linear dissipative systems and discrete-time optimal control problems. The closed-loop stability is established by exploiting a dissipative interconnection between the plant and optimizer dynamics.

\section{Optimization-based control as interconnection of dynamical and algebraic systems}\label{sec:interconnection}
In this section, we derive a scheme that treats optimization-based control as the interconnection of a dynamical system (the \textit{plant}) and an optimization-induced dynamical system (the \textit{optimizer dynamics}), as illustrated in Figure~\ref{fig:subopt}. The left box represents the system to be stabilized, while the right box corresponds to the virtual dynamics of the optimizer, e.g., arising as a gradient flow for the optimality system. Here, the optimizer dynamics is understood as a second dynamical control system where the initial value of the underlying dynamical control system is treated as an input and the corresponding output is a surrogate for the optimal control. 

\begin{figure}[htb]
\begin{center}
    \scalebox{0.9}{\begin{tikzpicture}[
  >=Latex,
  font=\small,
  block/.style={
    draw,
    thick,
    fill=gray!15,
    rounded corners=1pt,
    inner sep=.1cm,
    align=left,
    text width=5.2cm
  },
  wire/.style={thick, -Latex},
  lab/.style={font=\small, inner sep=1pt}
]

% --- Blocks ---
\node[block] (plant) {
\vspace*{-.3cm}
\begin{align*}
\dot x_p(t) &= -M_p(x_p(t)) + B_p u_p(t)\\
y_p(t) &= x_p(t)
\end{align*}
};

\node[block, right=2cm of plant, yshift=-1.5cm] (opt) {
\vspace*{-.2cm}
\begin{align*}
\varepsilon \dot w(t) &\in -M_{\mathrm{opt}}(w(t)) + \mathcal{B}u_\mathrm{opt}(t) \\
y_\mathrm{opt}(t) &= \mathcal{B}^* w(t)
\end{align*}
};

% --- Coordinates for clean routing ---
\coordinate (plantE) at ($(plant.east)+(0.2,0)$);
\coordinate (optN)   at ($(opt.north)+(0,0.2)$);
\coordinate (optW)   at ($(opt.west)+(-0.2,0)$);
\coordinate (plantS) at ($(plant.south)+(0,-0.2)$);

% --- y_p -> optimizer input (horizontal, then vertical) ---
\draw[wire]
(plant.east) -| node[lab, above]
{$u_\mathrm{opt}(t) = y_p(t)$}
(opt.north);

% --- optimizer output -> plant input (horizontal, then vertical) ---
\draw[wire]
(opt.west) -| node[lab, below]
{$u_p(t)= P_F(-\tfrac{1}{\alpha}B^*  y_\mathrm{opt}(t)+u_\mathrm{ref})$}
(plant.south);

\end{tikzpicture}}
\end{center}
\caption{Proposed MPC-type control-by-interconnection scheme. Although the optimizer dynamics is a maximally  monotone control system, the coupling is not structure preserving, i.e., the resulting closed loop system is solely driven by a quasi-m-monotone operator. }
\label{fig:subopt}
\vspace*{-.2cm}
\end{figure}
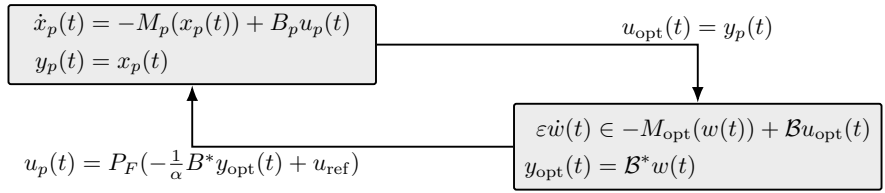

Here, as sketched in Figure~\ref{fig:subopt}, our aim is to model both the plant and the optimizer dynamics by the same system class and to perform a suitable interconnection in the sense of control systems. As we consider continuous-time formulations here, the optimizer dynamics evolve in a function space of time-dependent functions. Notably, the (autonomous) dynamics will be described by monotone  operators  ($M_p$ and $M_\mathrm{opt}$ in Figure~\ref{fig:subopt}) that are particularly well understood in operator theory and appear in a wide range of applications. This consideration paves the way towards suboptimal MPC of partial differential equations, a topic that is widely unexplored. While the plant model in Figure~\ref{fig:subopt} is formulated with single-valued operators, the optimality system is given by a set-valued operator. This will be particularly useful to allow for set-valued subdifferentials stemming from inequality constrained problems.

This section is organized as follows. Subsection~\ref{subsec:quasi_monotone} collects operator-theoretic preliminaries and a general well-posedness result for quasi-m-monotone evolution equations. Readers primarily interested in the control-theoretic developments may skip this subsection and refer back to it as needed. In Subsection~\ref{subsec:OS}, we formulate a constrained optimal control problem, derive the associated optimality system, and show that the resulting optimizer dynamics (right box in Figure~\ref{fig:subopt}) defines a maximally monotone control system. In Subsection~\ref{subsec:coupling}, we interconnect this system with the plant (left box in Figure~\ref{fig:subopt}) in an MPC fashion and establish well-posedness of the resulting closed loop using the results from Subsection~\ref{subsec:quasi_monotone}. Finally, Subsection~\ref{subsec:instant} discusses the limiting case of infinitely fast optimizer dynamics, which recovers a static MPC feedback law.

\subsection{Dynamical systems governed by quasi-m-monotone operators}\label{subsec:quasi_monotone}
In the following, we let $(X,\langle \cdot, \cdot \rangle_X)$ be a Hilbert space. This space will be the state space for the systems under consideration. For the plant dynamics (left box of Figure~\ref{fig:subopt}) this space is $X=\R^n$ while for the optimizer (right box of Figure~\ref{fig:subopt}), this space will be a consist of square integrable functions on $[0,T]$, that is, $X= L^2([0,T];\R^k)$ for suitably chosen $k\in \N$.

Let $M \subset X \times X$ be a subset of the product space $X \times X$. Using the same symbol, we associate with $M$ a \emph{set-valued operator} $M: X \rightrightarrows X$ where 
\begin{equation*}
     M(x) = \left\{y \in X \, \middle| \, (x,y) \in M \right\}
\end{equation*}
for all $x$ in the \emph{domain} $\dom M=\left\{x \in X \, \middle| \, M(x)\neq \emptyset \right\}$. We denote the set of zeros of a (possibly nonlinear) operator $M: X  \rightrightarrows X$ by
\[
\operatorname{zer} M = \{ x \in D(M) : 0 \in M(x)  \}.
\] The \emph{range} of $M$ is defined by $\ran M = \cup_{x \in \dom M} M(x)$.

We now define the central structural property of the governing operators following \cite[Ch. 3.1]{Barb10} and \cite{BausComb2011}.

\begin{definition}\label{def:monotonicity}
    A possibly nonlinear operator $M:X\rightrightarrows X$ is called 
\begin{enumerate}[label=(\roman*)]
    \item \emph{monotone} if
\begin{equation}\label{eq:mon}
\langle v - w, x - z \rangle_X \geq 0 \qquad \text{for all} \qquad x,z \in \dom{M}, v \in M(x), w \in M(z).
\end{equation}
\item \textit{maximally monotone} (abbreviated \emph{m-monotone}) if it is monotone and if it does not admit a proper monotone extension\footnote{Equivalently, $\operatorname{ran}(\lambda I + M) = X$ for some (and hence for all) $\lambda > 0$, see \cite{BausComb2011}.}. 
\item \emph{$\omega$-quasi-monotone} if there is $\omega \in \R$ such that $M+\omega I$ is monotone.
\item \emph{$\omega$-quasi-m-monotone} if there is $\omega \in \R$ such that $M+\omega I$ is m-monotone.
\end{enumerate}
\end{definition}

\noindent To denote the monotonicity, that is, \eqref{eq:mon} of the above definition, we will sometimes abbreviate
\begin{equation*}
    \langle M(x) - M(z), x-z\rangle_X \geq 0 \qquad \text{for all} \qquad x,z \in \dom{M}.
\end{equation*}
A prototypical example of a set-valued maximally monotone operator is the subdifferential $\partial g: X \rightrightarrows X$ associated with a proper, convex, and lower semicontinuous function $g: X \to \R \cup \{+\infty\}$. For $x \in X$, the subdifferential is defined by
\begin{equation*}
    \partial g (x) \coloneqq \left\lbrace y \in X \, \middle| \, \langle y, z-x \rangle_X \leq F(z)-F(x) \, \forall \, z \in X \right\rbrace.
\end{equation*}

\noindent We now consider the generalized homogeneous Cauchy problem 
\begin{align}\label{eq:homo_CP}
 \tfrac{\mathrm{d}}{\mathrm{d}t}x(t) + M(x(t)) + F(x(t)) \ni 0 \quad t \geq 0, \qquad x(0)=x_0
\end{align}
where $M$ is maximally monotone and $F:X\to X$. Let $x_0 \in \dom{M}$. Then, a function \(x : [0,\infty) \to X\) is a \textit{strong solution} of \eqref{eq:homo_CP} on \([0,\infty)\) if and only if $x(0)=x_0$, \(x\) is Lipschitz continuous on compact subsets of \([0,\infty)\), \(x\) is differentiable almost everywhere on \([0,\infty)\), and the differential inclusion of \eqref{eq:homo_CP} is satisfied almost everywhere, \cite{crandall1971generation}. 

We have the following generation theorem providing existence of strong solutions to \eqref{eq:homo_CP} assuming that the perturbation $F$ is globally Lipschitz. In our particular application to suboptimal MPC schemes (Figure~\ref{fig:subopt}), these perturbations arise due to the interconnection of two monotone systems.
\begin{theorem}\label{thm:homogeneous_Cauchy}
Let $M: X \rightrightarrows X$ be a maximally monotone operator and let $F: X \to X$ be globally Lipschitz continuous, that is, there is $L>0$ such that
\begin{equation*}
    \|F(x)-F(z)\|_X \leq L\|x-z\|_X \quad \mathrm{for\ all}\ x,z\in X.
\end{equation*}
Then, for every $x_0\in \dom{M}$ there exists a unique strong solution of~\eqref{eq:homo_CP}. In particular, the operator $M+F$ is the \textit{infinitesimal generator} of the semigroup given by the exponential formula
\begin{equation}
    S_{M+F}(t)x_0 = \lim_{n \to \infty} (I+\tfrac{t}{n}(M+F))^{-n}x_0.
\end{equation}
Moreover, the semigroup $S_{M+F}$ is $L$-quasi-contractive, i.e.,
\begin{equation}\label{eq:contraction_semigroup}
    \|  S_{M+F}(t)x- S_{M+F}(t)y\|_X \leq e^{Lt} \| x-y\|_X 
\end{equation}
for all $t \geq 0$, $x,y \in \overline{\dom{M}}$.
\end{theorem}
\begin{proof}
    We show that $M+F$ is $L$-quasi-m-monotone. Then, the claim follows directly from \cite[Thm. 2]{crandall1971generation} and \cite[Prop. 4.2]{Barb10}. To this end, as $F$ is globally Lipschitz, we compute
\begin{equation*}
    \langle F(x)-F(z), x-z \rangle_X \geq -\|F(x)-F(z)\|_X\,\|x-z\|_X \geq -L\|x-z\|_X^2
\end{equation*}
for all $x,z \in X$. Consequently,
\begin{equation*}
    \langle (M+F)(x)-(M+F)(z), x-z \rangle_X \geq -L\|x-z\|_X^2,
\end{equation*}
that is, $M+F$ is $L$-quasi-monotone.
Equivalently, $M+F+LI$ is monotone. The maximality of $M+(F+LI)$ follows from the fact that continuous, monotone and everywhere defined perturbations of maximally monotone operators are again maximally monotone, \cite[Cor. 25.5]{BausComb2011}.
\end{proof}

Last, we define the system class to model plant and optimizer dynamics. As sketched in Figure~\ref{fig:subopt}, the plant dynamics will indeed be governed by a single-valued operator, while for the optimizer dynamics we allow for set-valued operators.
\begin{definition}\label{def:monosys}
Let $X$ and $U$ be Hilbert spaces. Let $M:X\rightrightarrows X$ be a maximally monotone operator, and let $B \in L(U,X)$. Then we call
\begin{subequations}
    \label{eq:monotone_phs_def}
    \begin{align}
    \tfrac{\mathrm{d}}{\mathrm{d}t} x(t) &\in -M(x(t)) + Bu(t), \qquad x(0)=x_0,\label{eq:monotone_phs_def:state}\\
    y(t) &= B^* x(t) \label{eq:monotone_phs_def:output}
\end{align}
\end{subequations}
a \textit{maximally monotone control system}, abbreviated by $(M,B)$.
\end{definition}

\subsection{Optimality system as monotone operator equation}\label{subsec:OS}
In this subsection, we derive dynamical systems arising from gradient-type flows in optimal control. In this context, we consider the \textit{optimal control problem} 
\begin{align}\label{eq:oc_with_dynamics}
\begin{aligned}
\min_{(x,u)\in L^2(0,t_f; \R^n) \times L^2(0,t_f; \R^m)} \quad
& \int_0^{t_f} \left(
\tfrac{1}{2} \| x(\tau) - x_\mathrm{ref} \|_{\R^n}^2
+ \tfrac{\alpha}{2} \| u(\tau) - u_\mathrm{ref} \|_{\R^m}^2
\right) \,\mathrm{d}\tau \\
\text{s.t.} \quad
& \dot{x}(\tau) = Ax(\tau) + Bu(\tau), \quad x(0) = x_0, \\
& u \in F
\end{aligned}
\end{align}
with finite time horizon $t_f>0$, controlled equilibrium\footnote{We say $(x_\mathrm{ref}, u_\mathrm{ref})$ is a \textit{controlled equilibrium} of \eqref{eq:monotone_phs_def} if $-M(x_\mathrm{ref})+Bu_\mathrm{ref} = 0$. } $(x_\mathrm{ref}, u_\mathrm{ref})$ of $(M_p,B_p)$, regularization parameter $\alpha > 0$, $A \in \R^{n \times n}$, $B \in \R^{m \times n}$ and convex and closed control constraint set $F \subset L^2(0,t_f; \R^m)$ with $ u_\mathrm{ref} \in F$.

\medskip

To further analyze the optimal control problem, we reformulate \eqref{eq:oc_with_dynamics} as a constrained optimization problem in Hilbert spaces, using the linear operator 
\begin{equation*}
    \calC: L^2(0,t_f; \R^n) \times L^2(0,t_f; \R^m) \supset \dom{\calC} \to   L^2(0,t_f; \R^n)  \times \R^n
\end{equation*}
with 
\begin{equation}\label{eq:defC}
    \calC \begin{bmatrix}
        x \\ u
    \end{bmatrix} = \begin{bmatrix}
        \tfrac{\mathrm{d}}{\mathrm{d}\tau}x-Ax-Bu \\ x(0)
    \end{bmatrix}, \quad \dom{\calC} = H^1(0,t_f; \R^n) \times L^2(0,t_f; \R^m),
\end{equation}
where $H^1(0,t_f; \R^n) $ is the space of functions in $L^2(0,t_f; \R^n)$ whose weak derivative is also in $L^2(0,t_f; \R^n)$. This operator corresponding to the constraint has the particular structure that the part acting on the state is surjective. More precisely, the operator
\begin{align*}
    \mathcal{A} x = \begin{bmatrix}
        \tfrac{\mathrm{d}}{\mathrm{d}\tau}x-Ax \\ x(0)
    \end{bmatrix}, \quad  \calA : L^2(0,t_f;\R^n) \supset \dom{\calA} \coloneqq H^1(0,t_f; \R^n) \to L^2(0,t_f;\R^n)\times \R^n
\end{align*}
is surjective. In particular, for $(f,x_0)\in L^2(0,t_f;\R^n)\times \R^n$,
\begin{align*}
    \calA x = \begin{bmatrix}
        f\\
        x_0
    \end{bmatrix}  \quad \mathrm{if\ and\ only\ if\ } \quad x(t) = e^{tA}x_0 + \int_0^t e^{(t-s)A} f(s)\,\mathrm{d}s.
\end{align*}
This implies that also $\calC$ is surjective and hence has closed range.

We abbreviate the \textit{cost function} by $J: L^2(0,t_f; \R^n) \times L^2(0,t_f; \R^m) \to \R$ with
\begin{equation}
    J(x,u) = \int_0^{t_f} \left(
\tfrac{1}{2} \| x(\tau) - x_\mathrm{ref} \|_{\R^n}^2
+ \tfrac{\alpha}{2} \| u(\tau) - u_\mathrm{ref} \|_{\R^m}^2
\right) \,\mathrm{d}\tau
\end{equation}
To simplify notation we introduce the
\begin{equation*}
    \textit{primal variable } z=(x,u) \qquad \text{and} \qquad \textit{dual variable }  p=(\lambda, \lambda_0).
\end{equation*}
The proof of the subsequent theorem is given in \cite[Thm. 3.1]{preuster2026optimization}. Here, we denote by  $i_K: Y \to \R \cup \{+\infty\}$ the \textit{indicator function} of some subset $K$ of the Hilbert space $Y$. 
 \begin{equation*}
     i_K(y)\coloneqq \renewcommand{\arraystretch}{0.1}
 \begin{cases}
 0, & \text{if } y \in K,\\
 +\infty, & \text{otherwise}.
 \end{cases}
 \end{equation*}
\begin{theorem}
    There exists a unique optimal solution $z^\star = (x^\star, u^\star) \in \dom{\calC}$ of \eqref{eq:oc_with_dynamics}. Moreover, there exists a $p^\star \in \dom{\calC^*}$ such that 
\begin{equation}\label{eq:os}
    0 \in \begin{bmatrix}
        \stack{x^\star- x_\mathrm{ref}}{\alpha(u^\star- u_\mathrm{ref})} + \calC^* p^\star \vphantom{ \stack{ \{0\}}{\partial_{i_F}(u^\star)}} \\ -\calC  \stack{x^\star}{u^\star}+\stack{0}{x_0}
    \end{bmatrix} + \begin{bmatrix}
        \stack{ \{0\}}{\partial_{i_F}(u^\star)} \\ \{0\} \vphantom{\stack{x^\star}{u^\star}}
    \end{bmatrix} \eqqcolon \calA_{x_0} (z^\star, p^\star) + \calS  (z^\star, p^\star).
\end{equation}
\end{theorem}
  
In the following, we will reformulate this optimality system. First, we will give an alternative representation of the adjoint as a backwards-in-time differential equation. 
     
     As proven in \cite[Lem. 3.2]{gernandt2025port}, the Hilbert space adjoint  $\calC^*:L^2(0,t_f; \R^n)  \times \R^n \supset \dom{\calC^*} \to  L^2(0,t_f; \R^n) \times L^2(0,t_f; \R^m)$ is given by
        \begin{align*}
        \calC^* \begin{bmatrix}
            \lambda \\ \lambda_0
        \end{bmatrix} = \begin{bmatrix}
        -\frac{\mathrm{d}}{\mathrm{d}\tau} \lambda -A^* \lambda \\
        -B^* \lambda
    \end{bmatrix}
    \end{align*}
    with domain
    \begin{align*}
       \dom{\calC^*} = \dom{\calA^*} = \{(\lambda,\lambda_0)\in H^1(0,t_f; \R^n)  \times \R^n \,|\, \lambda(0)=\lambda_0 \wedge \lambda(t_f)=0\}.
    \end{align*}
The following lemma shows that the control $u$ can be eliminated from the optimality system \eqref{eq:os}, yielding a \textit{control-reduced optimality system}. To this end we use the \emph{projection} $P_C$ onto a closed convex subset $C$ of a Hilbert space $X$ and we recall that 
projection onto a closed convex set $P_C$ is \emph{firmly non-expansive}, that is
\begin{equation}\label{eq:P_F_firmly_nonexp}
         \langle P_C(x)-P_C(z),x-z\rangle_{X} \geq \| P_C(x)-P_C(z)\|_{X}^2
     \end{equation}
     for all $x,z \in X$, see e.g. \cite{BausComb2011}. %\cite[Lem. 3.1]{preuster2026stabilization}.
\begin{lemma}\label{lem:optimal_control}
    The tuple $((x^\star,u^\star), (\lambda^\star, \lambda_0^\star)) \in \dom{\calC} \times \dom{\calC^*}$ solves \eqref{eq:os} if and only if $(x^\star, (\lambda^\star, \lambda_0^\star)) \in H^1(0,t_f; \R^n) \times \dom{\calC^*}$ solve the \emph{reduced optimality system} 
    \begin{equation}\label{eq:red_OS}
   \begin{bmatrix}
        x^\star- x_\mathrm{ref} +( -\frac{\mathrm{d}}{\mathrm{d}\tau} \lambda^\star -A^* \lambda^\star)  \\
        \stack{- (\tfrac{\mathrm{d}}{\mathrm{d}\tau}x^\star -A x^\star) + B P_F(1/\alpha B^* \lambda^\star+u_\mathrm{ref})}{-x^\star(0)} +\stack{0}{x_0}
    \end{bmatrix}=0, \qquad \lambda^\star(0)= \lambda_0^\star,
\end{equation}
    and 
    \begin{equation}\label{eq:opt_control}
     u^\star = P_F(\tfrac{1}{\alpha}B^* \lambda^\star+u_\mathrm{ref}).
\end{equation}
\end{lemma}
\begin{proof}
   The proof is direct consequence of the fact that the resolvent of the subgradient of $i_F$ coincides with the projection $P_F$ of $u \in L^2(0,t_f; \R^m)$ onto the closed and convex set $F$ (see, e.g.\ \cite[Lem. 3.1]{preuster2026stabilization}): 
    \begin{align*}
        & 0 \in \alpha u^\star - \alpha u_\mathrm{ref} - B^* \lambda^\star + \partial i_F(u^\star) \\
        \Longleftrightarrow \hspace{1cm} & 0 \in  u^\star - u_\mathrm{ref} - \tfrac{1}{ \alpha}B^* \lambda^\star + \partial i_F(u^\star) \\
        \Longleftrightarrow \hspace{1cm} &  \tfrac{1}{ \alpha}B^* \lambda^\star + u_\mathrm{ref} \in  u^\star  + \partial i_F(u^\star) \\
        \Longleftrightarrow \hspace{1cm} & u^\star = (I+\partial i_F)^{-1}( \tfrac{1}{ \alpha}B^* \lambda^\star + u_\mathrm{ref}) = P_F( \tfrac{1}{ \alpha}B^* \lambda^\star + u_\mathrm{ref} ) .
    \end{align*}
\end{proof}
The following results reveals a monotonicity property of the operator governing the reduced optimality system \eqref{eq:red_OS}.
\begin{proposition}\label{prop:OS_operator_inj}
    For all $x_0 \in X$, the operator $M_{\mathrm{opt},x_0} : W \supset \dom{M_{\mathrm{opt},x_0}} \eqqcolon \dom{\calA} \times \dom{\calA^*} \to W$ with $W \coloneqq  L^2(0,t_f;\R^n) \times L^2(0,t_f; \R^n)  \times \R^n$ and 
    $$
    M_{\mathrm{opt},x_0}\begin{bmatrix}
       x \\ \stack{\lambda}{\lambda_0}
   \end{bmatrix} \coloneqq \begin{bmatrix}
        x^\star- x_\mathrm{ref} +\calA^* \lambda^\star  \\
        -\calA x^\star + \stack{B P_F(1/\alpha B^* \lambda^\star+u_\mathrm{ref})}{0} +\stack{0}{x_0}
    \end{bmatrix}
    $$ 
    is maximally monotone and injective. 
\end{proposition}
\begin{proof}
    The operator $M_{\mathrm{opt},x_0}$ is decomposed as
    \begin{equation}\label{eq:os_decomp}
         M_{\mathrm{opt},x_0}\begin{bmatrix}
       x \\ \stack{\lambda}{\lambda_0}
   \end{bmatrix} = \begin{bmatrix}
        x- x_\mathrm{ref}  \\
        \stack{B P_F(1/\alpha B^* \lambda+u_\mathrm{ref})}{0}
    \end{bmatrix} + \begin{bmatrix}
       0 &  \calA^*  \\
        -\calA  & 0\vphantom{ \stack{\lambda}{\lambda_0}}
    \end{bmatrix}\begin{bmatrix}
       x \\ \stack{\lambda}{\lambda_0}
   \end{bmatrix}  + \begin{bmatrix}
        0  \\
       \stack{0}{x_0} \vphantom{ \stack{\lambda}{\lambda_0}}
    \end{bmatrix}.
    \end{equation}
    Due to linearity and skew-adjointness of the second block operator of the above equation, we have  for all $w_i = (x_i,(\lambda_i,\lambda_{0i})) \in \dom{M_{\mathrm{opt},x_0}}$ that
    \begin{align}
    \nonumber
        & \langle M_{\mathrm{opt},x_0} (w_1) -M_{\mathrm{opt},x_0} (w_2),w_1-w_2 \rangle_W \\
        &= \langle x_1 - x_\mathrm{ref}- (x_2 - x_\mathrm{ref}),x_1-x_2 \rangle_{L^2(0,t_f;\R^n)} \nonumber\\
        & \hphantom{=} + \langle B P_F(\tfrac{1}{\alpha} B^* \lambda_1+u_\mathrm{ref})-B P_F(\tfrac{1}{\alpha} B^* \lambda_2+u_\mathrm{ref}) ,\lambda_1-\lambda_2 \rangle_{L^2(0,t_f;\R^n)} \nonumber\\
        & = \| x_1-x_2 \|_{L^2(0,t_f;\R^n)}^2 \label{eq:norm_estimate}\\
        &\hphantom{=}+ \alpha  \langle  P_F(\tfrac{1}{\alpha} B^* \lambda_1+u_\mathrm{ref})- P_F(\tfrac{1}{\alpha} B^* \lambda_2+u_\mathrm{ref}) , \tfrac{1}{\alpha}B^* \lambda_1-\tfrac{1}{\alpha}B^*\lambda_2 \rangle_{L^2(0,t_f;\R^m)} \nonumber\\
        &\geq \alpha  \langle  P_F(\tfrac{1}{\alpha} B^* \lambda_1+u_\mathrm{ref})- P_F(\tfrac{1}{\alpha} B^* \lambda_2+u_\mathrm{ref}) , \tfrac{1}{\alpha}B^* \lambda_1+u_\mathrm{ref}-(\tfrac{1}{\alpha}B^*\lambda_2 +u_\mathrm{ref})\rangle_{L^2(0,t_f;\R^m)} \nonumber \\
        & \stackrel
        {\eqref{eq:P_F_firmly_nonexp}}{\geq~~~~} \alpha \|  P_F(\tfrac{1}{\alpha} B^* \lambda_1+u_\mathrm{ref})- P_F(\tfrac{1}{\alpha} B^* \lambda_2+u_\mathrm{ref}) \|^2_{L^2(0,t_f;\R^m)} \nonumber\\
        & \geq 0\nonumber.
    \end{align}
The maximality is a direct consequence of the decomposition \eqref{eq:os_decomp} of $M_{\mathrm{opt},x_0}$ and into a skew-adjoint operator and a continuous everywhere-defined operator, \cite[Cor. 25.5]{BausComb2011}. To prove injectivity, let $M_{\mathrm{opt},x_0}(w_1)=M_{\mathrm{opt},x_0}(w_2)$ for some $w_1,w_2 \in \dom{M_{\mathrm{opt},x_0}}$. From the estimate \eqref{eq:norm_estimate}, we conclude that $ x_1 = x_2$. Hence, 
    \begin{equation*}
        \calC^* \left( \roundstack{\lambda_1}{\lambda_{01}} -\roundstack{\lambda_2}{\lambda_{02}} \right) =0
    \end{equation*}
    and by injectivity of $\calC^*$, \cite[Lem. 3.2]{gernandt2025port}, we obtain $w_1=w_2$.
\end{proof}
We now introduce the flow associated with the operator $M_{\mathrm{opt},x_0}$ that will model the associated optimization algorithm in the sense of a gradient flow. One may then consider the nonlinear evolution equation
\begin{equation*}
    \frac{\mathrm{d}}{\mathrm{d}t} w(t) = - M_{\mathrm{opt},x_0}(w(t)). 
\end{equation*}
Note that the preceding evolution equation can be also understood as a primal-dual gradient dynamics where the control is reduced from the Lagrangian, \cite{cherukuri2016asymptotic,cherukuri2017saddle}. 
Observe that 
\begin{equation*}
    M_{\mathrm{opt},x_0} = M_{\mathrm{opt},0} + \begin{bmatrix}
        0 \\ \stack{0}{x_0}
    \end{bmatrix}.
\end{equation*}

In view of our application in MPC (see Algorithms~\ref{alg:MPC} and~\ref{alg:subopt_MPC}), the initial value $x_0 \in X$ will play the role of an input to the optimizer dynamics from the plant system (see Figure~\ref{fig:subopt}). Consequently, we define 
the \textit{optimizer dynamics}
\begin{subequations}\label{eq:optimizer}
    \begin{align}
    \varepsilon \frac{\mathrm{d}}{\mathrm{d}t} \begin{pmatrix}
        x(t) \\ p(t)
    \end{pmatrix}  &= - M_{\mathrm{opt},0} \begin{pmatrix}
        x(t) \\ p(t)
    \end{pmatrix} - \begin{bmatrix}
        0 \\ \stack{0}{I}
    \end{bmatrix} u_\mathrm{opt}(t), \label{eq:optimizer:state}\\
    y_\mathrm{opt}(t) &= - \lambda_0(t),\label{eq:optimizer:output}
\end{align}
\end{subequations}
associated with the optimality system~\eqref{eq:red_OS}, where $\varepsilon > 0$ acts as a time-scaling parameter. Introducing $\varepsilon$ allows the optimizer and plant dynamics to evolve on different time scales in the continuous-time formulation. In particular, small values of $\varepsilon$ correspond to fast optimizer dynamics relative to the plant, whereas large values correspond to slower optimizer dynamics. Equivalently, by the chain rule,
\begin{equation*}
   \tfrac{\mathrm{d}}{\mathrm{d}t} w(\varepsilon t)
    = \varepsilon \tfrac{\mathrm{d}}{\mathrm{d}t} w(\varepsilon t),
\end{equation*}
so that $\varepsilon$ induces a rescaling (stretching or compression) of time.
The system~\eqref{eq:optimizer} indeed defines a maximally monotone control system as in Definition~\ref{def:monosys}. We define $\mathcal{B} \coloneqq [0,[0,I]]^\top : \R^n \to W$ to represent the input map. Note that~\eqref{eq:optimizer} defines a control system on the state space $W$ and the input/output space $\R^n$.

\subsection{Coupled optimizer-plant dynamics}\label{subsec:coupling}
As discussed in the beginning of this section, our aim is to define a stabilizing control (that respects some input constraints) for a continuous-time finite-dimensional control system governed by a monotone operator $M_p: \R^n \to \R^n$ and an input mapping $B_p \in \R^{m \times n}$. Precisely, consider the control system
\begin{subequations}\label{eq:MPC_plant}
    \begin{align}
    \tfrac{\mathrm{d}}{\mathrm{d}t} x_p(t) &= -M_p(x_p(t)) + B_pu_p(t), \qquad x_p(0)=x_{p0},\label{eq:MPC_plant:state}\\
    y_p(t) &= x_p(t) .\label{eq:MPC_plant:output}
\end{align}
\end{subequations}
Here we note that the output map is chosen as the identity, as the information flow from the plant to the optimal control problem in the standard MPC algorithm is done via the initial value, see Algorithm~\ref{alg:MPC}. We refer to \eqref{eq:MPC_plant}, abbreviated by $(M_p,B_p,I)$, as the \textit{plant}, as it represents the system being controlled. Note that this system does not belong to the class of maximally monotone control systems defined in Definition~\ref{def:monosys}, as the output operator is not the adjoint of the input operator.

We now interconnect the plant with the optimizer dynamics. In an MPC-fashion (see Sections~\ref{subsec:mpc_algo} and \ref{subsec:suboptmpc_intro} and Figure~\ref{fig:subopt}), we define the coupling condition 
\begin{equation}\label{eq:coupling}
    \begin{bmatrix}
        u_p(t) \\ u_\mathrm{opt}(t)
    \end{bmatrix} = \begin{bmatrix}
        P_F(\tfrac{1}{\alpha}B^* \lambda_0(t)+u_\mathrm{ref}) \\ x_p(t)
    \end{bmatrix} = \begin{bmatrix}
        P_F(-\tfrac{1}{\alpha}B^*  y_\mathrm{opt}(t)+u_\mathrm{ref}) \\ y_p(t)
    \end{bmatrix} .
\end{equation}
Moreover, the state space of the interconnected system is simply the Cartesian product of the individual state spaces of \eqref{eq:optimizer} and \eqref{eq:MPC_plant}, that is, $V \coloneqq \R^n \times W$. 
Depending on the level of abstraction, we will abbreviate the state of the coupled system via
\begin{equation*}
    v(t) \coloneqq (x_p(t),w(t)) \coloneqq (x_p(t), x(t) , (\lambda(t), \lambda_0(t))) .
\end{equation*}
The coupling condition \eqref{eq:coupling} leads to the autonomous continuous-time \textit{coupled optimizer-plant dynamics}
\begin{align}\label{eq:optimizer_plant_epsi_ge0}
\begin{bmatrix}
    I &0 \\0 & \varepsilon
\end{bmatrix}\frac{\mathrm{d}}{\mathrm{d}t} \begin{pmatrix}
    x_p(t) \\ w(t)
\end{pmatrix}  &=-\begin{bmatrix}
    M_p(x_p(t)) \vphantom{ BP_F(\tfrac{1}{\alpha}B^*  \lambda_0(t))} \\ M_{\mathrm{opt},0}(w(t))  \vphantom{ \stack{0}{\stack{0}{x_p(t)}}}
\end{bmatrix} + \begin{bmatrix}
      BP_F(\tfrac{1}{\alpha}B^*  \lambda_0(t)+u_\mathrm{ref}) \\ \stack{0}{\stack{0}{-x_p(t)}}
\end{bmatrix}
\end{align}
which is written compactly as
\begin{align}\label{eq:optimizer_plant}
\frac{\mathrm{d}}{\mathrm{d}t} v(t)  &= - M(v(t)) -F(v(t)), \qquad v(0)=v_0, \qquad t\geq 0. 
\end{align}
with the diagonal $M: V \supset \dom{M} \to V$ defined by 
\begin{equation}\label{eq:def_M}
    M \begin{bmatrix}
       x_p \\ \begin{bmatrix}
           x \\ \stack{\lambda}{\lambda_0}
       \end{bmatrix}
    \end{bmatrix}  \coloneqq \begin{bmatrix}
        M_p(x_p) \\ \tfrac{1}{\varepsilon} M_{\mathrm{opt},0} \begin{bmatrix}
           x \\ \stack{\lambda}{\lambda_0}
       \end{bmatrix} 
    \end{bmatrix}, \qquad \dom{M}= \R^n \times \dom{M_{\mathrm{opt},0}}
\end{equation}
and the interconnection term $F: V \to V$ given by
\begin{equation}\label{eq:def_F}
    F \begin{bmatrix}
       x_p \\ \begin{bmatrix}
           x \\ \stack{\lambda}{\lambda_0}
       \end{bmatrix}
    \end{bmatrix}  \coloneqq - \begin{bmatrix}
        BP_F(\tfrac{1}{\alpha}B^*  \lambda_0+u_\mathrm{ref}) \\ - \tfrac{1}{\varepsilon} \begin{bmatrix}
           0 \vphantom{x} \\ \stack{0\vphantom{\lambda_0}}{x_p}
       \end{bmatrix} 
    \end{bmatrix},
\end{equation}

The following theorem establishes well-posedness of the coupled optimizer-plant system \eqref{eq:optimizer_plant}.
\begin{theorem}\label{thm:optimizer_plant}
The operator $M$ given by \eqref{eq:def_M} is m-monotone and $F: V \to V$ as in \eqref{eq:def_F} is globally Lipschitz continuous with Lipschitz bound 
\begin{equation*}
    L \coloneqq \sqrt{2}\max(\tfrac{2}{\alpha}\|B\|_{L(U,X)}, \tfrac{1}{\varepsilon}).
\end{equation*}
Moreover, for every $v_0\in \dom{M}$, there exists a unique strong solution of~\eqref{eq:optimizer_plant}. In particular, the operator $M+F$ is the \textit{infinitesimal generator} of the semigroup given by the exponential formula
\begin{equation}\label{eq:optimizer_plant_semigroup}
    S_{M+F}(t)v_0 = \lim_{n \to \infty} (I+\tfrac{t}{n}(M+F))^{-n}v_0.
\end{equation}
\end{theorem}
\begin{proof}
    The operator $M$ is m-monotone because it is a block-diagonal operator composed of m-monotone operators: The upper block $M_p$ is m-monotone by assumption and the bottom block given by $\tfrac1\varepsilon M_{\mathrm{opt},0}$ is m-monotone since it is a continuous, monotone and everywhere defined perturbation of a linear skew-adjoint operator. We proceed with the global Lipschitz bound on the perturbation $F$. To this end, observe that firm non-expansivity of the projection (see \eqref{eq:P_F_firmly_nonexp}) yields 
    \begin{align*}
        \| P_F(x)-P_F(z) \|^2& \leq \langle  P_F(x)-P_F(z), x-z \rangle  \leq  \| P_F(x)-P_F(z) \| \|x-z\|
    \end{align*}
    which implies $$\| P_F(x)-P_F(z) \| \leq \|x-z\|$$ for all $x,z \in L^2(0,t_f; \R^m)$, where the norm subscript $L^2(0,t_f; \R^m)$ is omitted for readability. The product space $V=\R^n \times W$ of two Hilbert spaces may be endowed with either 
    \begin{equation*}
        \| (x,y)\|_1 \coloneqq \| x\|_{\R^n} + \|y\|_W \qquad \text{or} \qquad  \| (x,y)\|_2 \coloneqq (\| x\|_{\R^n}^2 + \|y\|_W^2)^{1/2},
    \end{equation*}
    Clearly, both norms are equivalent and are related via the estimate
    \begin{equation*}
        \| (x,y)\|_2  \leq \| (x,y)\|_1 \leq  \sqrt{2} \| (x,y)\|_2 
    \end{equation*}
    for all $(x,y) \in V$. Consequently, let $v_i = (x_{p}^i, x^i, (\lambda^i, \lambda_0^i)) \in V, i=1,2$, be arbitrary and consider
    \begin{align*}
         \|F(v_1) - F(v_2)\|_2 & \leq  \|F(v_1) - F(v_2)\|_1 \\
         & = \|   BP_F(\tfrac{1}{\alpha}B^*  \lambda_0^1+u_\mathrm{ref}) -   BP_F(\tfrac{1}{\alpha}B^*  \lambda_0^2+u_\mathrm{ref}) \|_{\R^n} + \tfrac{1}{\varepsilon} \| x_p^1 -x_p^2 \|_{\R^n} \\ 
         & \leq \|B\|_{L(U,\R^n)}  \| P_F(\tfrac{1}{\alpha}B^*      \lambda_0^1+u_\mathrm{ref}) \hspace{-0.5mm} -\hspace{-0.5mm}  P_F(\tfrac{1}{\alpha}B^*  \lambda_0^2+u_\mathrm{ref}) \|_U \!+ \!\tfrac{1}{\varepsilon} \| x_p^1 -x_p^2 \|_{\R^n} \\
         & \leq \|B\|_{L(U,\R^n)}  \| \tfrac{1}{\alpha}B^*  \lambda_0^1 - \tfrac{1}{\alpha}B^*  \lambda_0^2 \|_U + \tfrac{1}{\varepsilon} \| x_p^1 -x_p^2 \|_{\R^n} \\
         & \leq \tfrac{1}{\alpha} \|B\|_{L(U,\R^n)}\|B^*\|_{L(\R^n,U)}   \|  \lambda_0^1 - \lambda_0^2 \|_{\R^n} + \tfrac{1}{\varepsilon} \| x_p^1 -x_p^2 \|_{\R^n} \\
         &\leq \tfrac{L}{\sqrt{2}} \|v_1 - v_2\|_1 \\
         & \leq L\|v_1 - v_2\|_2.
    \end{align*}
This proves that $F$ is globally Lipschitz and hence the assertion is an immediate consequence of Theorem~\ref{thm:homogeneous_Cauchy}.
\end{proof}
\subsection{Instantaneous optimization via infinitely fast optimizer dynamics}\label{subsec:instant}
The interconnection depicted in Figure~\ref{fig:subopt} and given by \eqref{eq:optimizer_plant_epsi_ge0} corresponds to a suboptimal MPC scheme (Algorithm~\ref{alg:subopt_MPC}), as the optimizer dynamics are not necessarily at steady state such that the optimality condition is not satisfied.  This is in contrast to the classical MPC approach described in Algorithm~\ref{alg:MPC}. The MPC-feedback law assigns to each current state $x_p(t)$ the optimal control $u^\star(t)=\mu(x_p(t))$ at time instance $t$; that is, it is assumed that the optimal solution is immediately available and characterized by an algebraic relation encoded by the optimality condition. In the spirit of the coupled optimizer-plant dynamics \eqref{eq:optimizer_plant}, see also Figure \ref{fig:subopt}, we may understand this as the limit case $\varepsilon=0$. 

More precisely, the coupled system is subject to the differential-algebraic equation (DAE)
\begin{align}\label{eq:optimizer_plant_DAE}
\begin{bmatrix}
    I & 0\\ 0& 0
\end{bmatrix}\frac{\mathrm{d}}{\mathrm{d}t} \begin{pmatrix}
    x_p(t) \\ w(t)
\end{pmatrix}  &=-\begin{bmatrix}
    M_p(x_p(t)) \vphantom{ BP_F(\tfrac{1}{\alpha}B^*  \lambda_0(t))} \\ M_{\mathrm{opt},0}(w(t))  \vphantom{ \stack{0}{\stack{0}{x_p(t)}}}
\end{bmatrix} + \begin{bmatrix}
      BP_F(\tfrac{1}{\alpha}B^*  \lambda_0(t)+u_\mathrm{ref}) \\ \stack{0}{\stack{0}{-x_p(t)}}
\end{bmatrix}.
\end{align}
This system is obtained from \eqref{eq:optimizer_plant} in the singular limit $\varepsilon \to 0$, in which the optimizer dynamics evolves on an infinitely fast time scale. A schematic illustration of this continuous-time coupling is given in Figure~\ref{fig:subopt_DAE}.
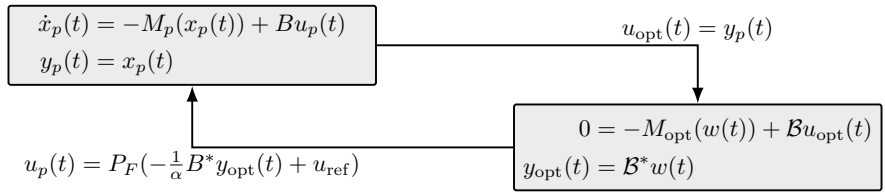
\begin{figure}[htb]
  \begin{center}  \scalebox{0.9}{\begin{tikzpicture}[
  >=Latex,
  font=\small,
  block/.style={
    draw,
    thick,
    fill=gray!15,
    rounded corners=1pt,
    inner sep=.1cm,
    align=left,
    text width=5.2cm
  },
  wire/.style={thick, -Latex},
  lab/.style={font=\small, inner sep=1pt}
]

% --- Blocks ---
\node[block] (plant) {
\vspace*{-.3cm}
\begin{align*}
\dot x_p(t) &= -M_p(x_p(t)) + B u_p(t)\\
y_p(t) &= x_p(t)
\end{align*}
};

\node[block, right=2cm of plant, yshift=-1.5cm] (opt) {
\vspace*{-.2cm}
\begin{align*}
0 &= -M_{\mathrm{opt}}(w(t)) + \mathcal{B}u_\mathrm{opt}(t) \\
y_\mathrm{opt}(t) &= \mathcal{B}^* w(t)
\end{align*}
};

% --- Coordinates for clean routing ---
\coordinate (plantE) at ($(plant.east)+(0.2,0)$);
\coordinate (optN)   at ($(opt.north)+(0,0.2)$);
\coordinate (optW)   at ($(opt.west)+(-0.2,0)$);
\coordinate (plantS) at ($(plant.south)+(0,-0.2)$);

% --- y_p -> optimizer input (horizontal, then vertical) ---
\draw[wire]
(plant.east) -| node[lab, above]
{$u_\mathrm{opt}(t) = y_p(t)$}
(opt.north);

% --- optimizer output -> plant input (horizontal, then vertical) ---
\draw[wire]
(opt.west) -| node[lab, below]
{$u_p(t)= P_F(-\tfrac{1}{\alpha}B^*  y_\mathrm{opt}(t)+u_\mathrm{ref})$}
(plant.south);

\end{tikzpicture}}
\end{center}
\caption{Coupling of dynamic plant with the algebraic optimality system as a differential-algebraic system.}
\label{fig:subopt_DAE}

\vspace*{-.2cm}
\end{figure}

Above in \eqref{eq:optimizer_plant_DAE} we observe that the optimizer block appears as an algebraic constraint. We will now show that one may also resolve this block using the properties of the optimality system. Precisely, from the injectivity of $M_{\mathrm{opt},0}$, see Proposition \ref{prop:OS_operator_inj}, we conclude that $M_{\mathrm{opt},0}^{-1}:\ran{ M_{\mathrm{opt},0}} \to W$ exists. Denote by
\begin{equation*}
    \pi_{\R^n} : W \to \R^n, \qquad 
    \pi_{\R^n}(x,(\lambda,\lambda_0)) = \lambda_0,
\end{equation*}
the projection onto the third component of $W$. 
Accordingly, the adjoint operator $\pi_{\R^n}^* : \R^n \to W$ embeds an element $x_0 \in \R^n$ into $W$ by setting the remaining components to zero. Together with Lemma \ref{lem:optimal_control}, we observe that the MPC-feedback law enjoys
%\textcolor{purple}{Adjoint of $\pi$ rather than transposed in the formula below?}
\begin{equation*}
    \mu(x_p(t)) \coloneqq P_F(\tfrac{1}{\alpha}B^* \pi_{\R^n} M_{\mathrm{opt},0}^{-1}(-\pi_{\R^n}^* x_p(t))+u_\mathrm{ref}).
\end{equation*}
This feedback leads to the closed-loop system 
\begin{equation*}
     \tfrac{\mathrm{d}}{\mathrm{d}t} x_p(t) = -M_p(x_p(t)) + B_pP_F(\tfrac{1}{\alpha}B^* \pi_{\R^n} M_{\mathrm{opt},0}^{-1}(-\pi_{\R^n}^* x_p(t))+u_\mathrm{ref}).
\end{equation*}
Note that these closed-loop dynamics are usually mainly of theoretical interest. In particular applications, the inverse appearing above is evaluated pointwise by solving an optimal control problem for each fixed initial state.

\section{Sample and hold in MPC: numerical splitting of coupled optimizer-plant dynamics} \label{sec:sampleholdMPC}
Having outlined the main idea of this section, we proceed as follows. In Subsection~\ref{subsec:lietrotter} we begin with a brief recapitulation of the Lie--Trotter splitting. Next, in Subsection~\ref{subsec:subopt_MPC} we apply this splitting to the coupled optimizer-plant dynamics \eqref{eq:optimizer_plant} for strictly positive time-scaling parameters and establish a one-to-one correspondence with the suboptimal MPC approach described in Algorithm~\ref{alg:subopt_MPC}. Finally, in Subsection~\ref{subsec:mpc_trotter_dae} we examine the relation between a Lie--Trotter-type splitting of the DAE \eqref{eq:optimizer_plant_DAE} and the classical MPC approach of Algorithm~\ref{alg:MPC}.
\subsection{Lie--Trotter splitting}\label{subsec:lietrotter}
We briefly recall the Lie--Trotter splitting as an approximation method for flows of dynamical systems from \cite{blanes2024splitting}. Let a dynamical system 
\begin{equation}\label{eq:auto_dyn_LT}
    \frac{\mathrm{d}}{\mathrm{d}t} v(t) = F(v(t)) = F_1(v(t)) + F_2(v(t)), \qquad v(0)=v_0,
\end{equation}
on a Hilbert space $V$ be given for some functions $F$, $F_1$, and $F_2$.
We assume that not only the flow associated with \eqref{eq:auto_dyn_LT} exists, but also the flows corresponding to the subproblems
\begin{equation*}
    \frac{\mathrm{d}}{\mathrm{d}t} v(t) = F_k(v(t)),\qquad v(0)=v_0, \qquad k = 1,2,
\end{equation*}
are well-defined. For $t \in [0, \infty)$, we denote these flows by $\varphi_t: V \to V$ and $\varphi^k_t: [0, \infty) \times V \to V$, $k=1,2$, respectively.
The \textit{Lie--Trotter splitting} approximates the exact flow $\varphi_h$ by
\begin{equation*}
    \varphi_h(v_0) \approx \varphi_h^2 \circ \varphi_h^1(v_0).
\end{equation*}
Here, $h>0$ denotes the splitting step size. To approximate the solution of \eqref{eq:auto_dyn_LT} on $[0,\infty)$, we partition the time interval into equidistant subintervals of length $h$. More precisely, we define $ t_n \coloneqq nh, n \in \mathbb{N}_0,$ with $t_0 = 0$, and successively apply the composition
$\varphi_h^2 \circ \varphi_h^1$ on each interval $[t_n,t_{n+1}]$.
The procedure is summarized in Algorithm~\ref{alg:lt}.
\begin{algorithm}
\caption{Lie--Trotter splitting.}\label{alg:lt}
\begin{algorithmic}[1]
\Statex \textbf{Given}: Initial state $v_0 = v(0)$, sampling time $h>0$.
\For{each $n = 0,1,2,\dots$}
\State Solve $\dot{v}_1 = F_1(v_1)$ with $v_1(t_n) = v_n$, on $[t_n, t_{n+1}]$.
\State Set $v_{n+1/2} = v_1(t_{n+1})$.
\State Solve $\dot{v}_2 = F_2(v_2)$ with $v_2(t_n) = v_{n+\frac{1}{2}}$, on $[t_n, t_{n+1}]$.
\State Set $x_{n+1} = y_2(t_{n+1})$.
\EndFor
\end{algorithmic}
\end{algorithm}

When it comes to numerical implementations, we have to discretize the flows $\varphi^k$ by appropriately chosen schemes for a step-size $h>0$ such as 
\begin{align}
    v_{n+1} &= \psi_h^{ke}(u_n) \coloneqq v_n + h F_k(v_n), 
    && \text{(explicit Euler)}, \label{eq:expl_euler}\\
    v_{n+1} &= \psi_h^{ki}(u_n) \coloneqq v_n + h F_k(v_{n+1}), 
    && \text{(implicit Euler)}.\label{eq:impl_euler}
\end{align}
for $k=1,2$.

\subsection{Suboptimal MPC is a Lie--Trotter splitting of interconnected dynamical systems}\label{subsec:subopt_MPC}
In this part, we apply a hybrid Lie--Trotter splitting to the coupled optimizer-plant dynamics \eqref{eq:optimizer_plant_epsi_ge0} and show that this corresponds to the suboptimal MPC algorithm~\ref{alg:subopt_MPC}. To this end, recall the continuous-time coupled optimizer-plant dynamics \eqref{eq:optimizer_plant_epsi_ge0} which, upon dividing the second block equation by $\varepsilon$ is
\begin{align}\label{eq:optimizer_plant_splitting}
\frac{\mathrm{d}}{\mathrm{d}t} \begin{pmatrix}
    x_p(t) \\ w(t)
\end{pmatrix}  &=-\begin{bmatrix}
    M_p(x_p(t)) \vphantom{ BP_F(\tfrac{1}{\alpha}B^*  \lambda_0(t))} \\ \tfrac{1}{\varepsilon}M_{\mathrm{opt},0}(w(t))  \vphantom{ \stack{0}{\stack{0}{x_p(t)}}}
\end{bmatrix} + \begin{bmatrix}
      BP_F(\tfrac{1}{\alpha}B^*  \lambda_0(t)+u_\mathrm{ref}) \\ \stack{0}{\stack{0}{-\tfrac{1}{\varepsilon}x_p(t)}}
\end{bmatrix}
\end{align}
For this fully coupled system, we have established in Theorem~\ref{thm:optimizer_plant} the existence of a suitable flow.
More precisely, let $S_{M+F}: V \to V$ be the nonlinear semigroup given by~\eqref{eq:optimizer_plant_semigroup} associated to the coupled optimizer-plant dynamics~\eqref{eq:optimizer_plant_epsi_ge0}. For $t \in [0, \infty)$, the flow $\varphi_t: V \to V$ is defined by 
\begin{equation*}
    \varphi_t(v_0) \coloneqq S_{M+F}(t)v_0.
\end{equation*}
To apply now a Lie--Trotter splitting to the system~\eqref{eq:optimizer_plant_epsi_ge0}, we do not decompose the system into the sum of diagonal and anti-diagonal part (that is, $M$ and $F$ as in \eqref{eq:optimizer_plant}), but rather split along the system interfaces, that is, into the two boxes of Figure~\ref{fig:subopt}. More precisely, we consider a decomposition of \eqref{eq:optimizer_plant_splitting} row-wise into
\begin{align}\label{eq:optimizer_plant_dec}
\frac{\mathrm{d}}{\mathrm{d}t} \begin{pmatrix}
    x_p(t) \\ w(t)
\end{pmatrix}  &=F_1\begin{pmatrix}
    x_p(t) \\ w(t)
\end{pmatrix} +F_2\begin{pmatrix}
    x_p(t) \\ w(t)
\end{pmatrix} , %\qquad w(0)=v_0, \qquad t\geq 0. 
\end{align}
where 
\begin{equation*}
     F_1\begin{pmatrix}
    x_p \\ w
\end{pmatrix}  \coloneqq \begin{bmatrix}
    -M_p(x_p) +  BP_F(\tfrac{1}{\alpha}B^*  \lambda_0+u_\mathrm{ref}) \\ 0 \vphantom{\stack{0}{\stack{0}{x_p}}}
\end{bmatrix} 
\end{equation*}
and 
\begin{equation*}
F_2\begin{pmatrix}
    x_p \\ w
\end{pmatrix} \coloneqq \begin{bmatrix}
   0 \vphantom{\tfrac{1}{\alpha}B^*} \\ -\tfrac{1}{\varepsilon} M_{\mathrm{opt},0}(w) -\tfrac{1}{\varepsilon} \stack{0}{\stack{0}{x_p}}
\end{bmatrix}.
\end{equation*}
It is a direct consequence of Theorem~\ref{thm:optimizer_plant} that also the flows associated with the subproblems 
\begin{equation*}
    \frac{\mathrm{d}}{\mathrm{d}t} \begin{pmatrix}
    x_p(t) \\ w(t)
\end{pmatrix}  = F_i\begin{pmatrix}
    x_p(t) \\ w(t)
\end{pmatrix} , \qquad i=1,2,
\end{equation*}
exist. We denote these flows by $\varphi_t^i: V \to V$ where
\begin{equation*}
    \varphi^i_t(v_0) \coloneqq S_{F_i}(t)v_0.
\end{equation*}
As explained in the preceding subsection, the Lie--Trotter splitting approximates the exact flow $\varphi$ by
\begin{equation*}
    \varphi_h(v_0) \approx \varphi_h^2 \circ \varphi_h^1(v_0).
\end{equation*}
Observe that in Step~4 of the suboptimal MPC approach of Algorithm~\ref{alg:subopt_MPC}, the optimal solution of the underlying optimal control problem \eqref{eq:red_OS} is approximated by an iterative solver. As an example, we consider the case where the iterative solver is given by
the \textit{proximal-point algorithm}, terminated after $j \in \mathbb{N}, j\geq 1$ iterations. The iterates are defined by
\begin{equation}\label{eq:prox_point}
     w^{(l+1)} = (I + \tfrac{h}{j\varepsilon} M_{\mathrm{opt},x_0})^{-1} w^{(l)}, \qquad l = 0, \ldots, j-1
\end{equation}
where $M_{\mathrm{opt},x_0}$ denotes the operator associated with the
optimality system \eqref{eq:red_OS}. We refer to \cite[Thm.~23.41]{BausComb2011} for the
convergence theory of the proximal-point method. We show that this method can be interpreted as an implicit Euler numerical approximation \eqref{eq:impl_euler} of $\varphi_t^2$, the exact flow associated with the optimizer dynamics (with trivial dynamics in the plant component). More general optimization schemes will be covered in Section~\ref{sec:opt}. Since, in real-world applications, the evolution of the plant is usually governed by the physical system and not by a numerical approximation, we treat $\varphi_t^1$ as an exact flow, see Step~1 of Algorithm \ref{alg:subopt_MPC}. Accordingly, we consider a Lie--Trotter step operator $\Psi_h^{j} : V \to V$ for some $j\in\mathbb{N}$ with $j\geq 1$ that is defined by
\begin{equation}\label{eq:one-step}
    \Psi_h^j\coloneqq  (\psi_{h/j}^{2i})^j \circ \varphi_h^1,
\end{equation}
where $\psi_{h/j}^{2i}$ denotes one implicit Euler step for the $F_2$-dynamics with step size $h/j$. In other words, we first evolve exactly under $F_1$ over a time interval of length $h$, and subsequently approximate the flow of $F_2$ over the same interval by $j$ implicit Euler substeps of size $h/j$, cf.~\eqref{eq:impl_euler}.
\begin{algorithm}[htb]
\caption{Lie--Trotter splitting with $j$ implicit Euler substeps.}
\label{alg:Lie_Trotter_subopt_MPC}
\begin{algorithmic}[1]
\Statex \textbf{Given}: Initial state $v_0 = v(0)$, sampling time $h>0$, truncation index $j\in \N$.
\For{each $n = 0,1,2,\dots$}
\State Solve $\dot{v}_1 = F_1(v_1)$ with $v_1(t_n) = v_n$, on $[t_n, t_{n+1}]$. \Comment{Propagate plant}
\State Set $v_{n+\frac{1}{2}} = v_1(t_{n+1})$.
\State Set $z^{(0)} = v_{n+\frac{1}{2}}$.
\For{$l = 0,\dots,j-1$} \Comment{Proximal gradient/imp.\ Euler for optimizer}
    \State $z^{(l+1)} = (I - \frac{h}{j} F_2)^{-1}z^{(l)}$
\EndFor
\State Set $v_{n+1} = z^{(j)}$.
\EndFor
\end{algorithmic}
\end{algorithm}
Clearly, other numerical approximations of the optimizer flow $\varphi_h^2$ can be employed as well, for instance based on the explicit Euler scheme \eqref{eq:expl_euler} or more advanced explicit-implicit schemes; these closely correspond to optimization schemes as illustrated in Section~\ref{sec:opt}. In what follows, we prove that Algorithm \ref{alg:Lie_Trotter_subopt_MPC}, coincides with the suboptimal MPC Algorithm~\ref{alg:subopt_MPC} with proximal-point iterative solver for the optimal control problem.
\begin{theorem}\label{thm:lt_prox}
Consider system \eqref{eq:control_system} with $f=-M_p$ and $g=B_p$. Fix a sampling time $h > 0$, and set a time-scaling $\varepsilon > 0$. Denote by $\varphi_t^1, \varphi_t^2$ the flows associated with the subproblems governed by $F_1,F_2$, respectively. For $v_n = (x_n, w_n) \in V$, define the one-step update 
\begin{equation}\label{eq:subopt_correspondence}
    v_{n+1} = \Psi_h^j(v_n),
\end{equation}
where $\Psi_h^j$ is given by \eqref{eq:one-step}. Then, the sequence $(v_n)_{n \ge 0}$ generated by the recursion \eqref{eq:subopt_correspondence} coincides with the sequence produced by Algorithm~\ref{alg:subopt_MPC} with $\tilde{f}=A, \tilde{g}=B$.
\end{theorem}
\begin{proof}
    By definition of the flow $\varphi_h^1$, we have that $v_{n+1/2}$ is the solution of the differential equation
    \begin{align}\label{eq:F_1}
        \frac{\mathrm{d}}{\mathrm{d}t} \begin{pmatrix}
    x_p(t) \\ w(t)
\end{pmatrix}  =  \begin{bmatrix}
    -M_p(x_p(t)) +  BP_F(\tfrac{1}{\alpha}B^*  \lambda_0(t)+u_\mathrm{ref}) \\ 0
\end{bmatrix}, \qquad v(t_n)=v_n 
    \end{align}
    at time $t_{n+1}$. Clearly, the second line of \eqref{eq:F_1} yields that the $W$-component of $\varphi_h^1(v_n)$ coincides with the initial value $w_n=(x_n, (\lambda_n,\lambda_{0n}))$. Consequently, \eqref{eq:F_1} reduces to 
     \begin{align}\label{eq:F_1_reduced}
        \frac{\mathrm{d}}{\mathrm{d}t} 
    x_p(t)= -M_p(x_p(t)) +  BP_F(\tfrac{1}{\alpha}B^*  \lambda_{0n}+u_\mathrm{ref}) 
\qquad x_p(t_n)=x_{pn}, \quad t \in [t_n, t_{n+1}].
    \end{align}
    The preceding differential equation has a unique solution $x_p(t)$ on $[t_n,t_{n+1}]$ and the end point evaluation $x_p(t_{n+1})$ coincides with the current state measurement in the Step~4 of Algorithm~\ref{alg:subopt_MPC} under the constant suboptimal control $\mu(x_{pn})=P_F(\tfrac{1}{\alpha}B^*  \lambda_{0n}+u_\mathrm{ref})$. Accordingly, the flow $\varphi_t^2$ denotes the unique solution of 
    \begin{align}\label{eq:F_2}
        \frac{\mathrm{d}}{\mathrm{d}t} \begin{pmatrix}
    x_p(t) \\ w(t)
\end{pmatrix}  =  \begin{bmatrix}
    0 \\  -\tfrac{1}{\varepsilon} M_{\mathrm{opt},0}(w(t)) -\tfrac{1}{\varepsilon} \stack{0}{\stack{0}{x_p(t)}}
\end{bmatrix}, \qquad v(t_n)=v_{n+1/2},
    \end{align}
    for  $t \in [t_n, t_{n+1}].$
The Steps 5-7 in Algorithm \ref{alg:Lie_Trotter_subopt_MPC} approximate this flow numerically via the implicit Euler scheme: set $z^{(0)} = v_{n+\frac{1}{2}}$ and compute for $l = 0,\dots,j-1$ the solution of the equation 
\begin{equation}\label{eq:impl_Euler_for_F_2}
    z^{(l+1)} = z^{(l)} + \frac{h}{j} F_2 \bigl(z^{(l+1)}\bigr) .   
\end{equation}
Write $z^{(l)}=(x_p^{(l)}, w^{(l)}) \in \R^n \times W$. Since $F_2$ maps everything to zero in the $\R^n$-component, see also \eqref{eq:F_2}, the iterates $x_p^{(l)}$ coincide with the initial value $x_{pn}$ for all $l \in \{0,\dots,j-1\}$. Consequently, \eqref{eq:impl_Euler_for_F_2} reduces to 
\begin{align}\label{eq:F_2_reduced}
       w^{(l+1)} = w^{(l)} - \frac{h}{j\varepsilon}  M_{\mathrm{opt},0}\bigl(w^{(l+1)}\bigr) - \frac{h}{j\varepsilon}  \stack{0}{\stack{0}{x_{pn}}} =  w^{(l)} - \frac{h}{j\varepsilon}  M_{\mathrm{opt},x_{pn}}\bigl(w^{(l+1)}\bigr) .
    \end{align}
    We rewrite this update rule via \eqref{eq:prox_point} with $x_0 = x_{pn}$,
    which is exactly the proximal-point algorithm.
\end{proof}

\subsection{MPC is a Lie--Trotter-like splitting of dynamical-algebraic systems}\label{subsec:mpc_trotter_dae}
In this subsection, we consider the DAE \eqref{eq:optimizer_plant_DAE}, that is, we assume that the optimality system is solved instantaneously as in classical MPC. In this case, the splitting scheme as in the preceding subsection may not be employed as the second system does not give rise to a flow as it is modeled by an algebraic equation. To resolve this issue, we define the set-valued operator $M_0 : V \rightrightarrows V$ with 
   \begin{equation*}
       M_0 = \left\{ \left(\begin{bmatrix}
            x_p \vphantom{M_p(x_p)}  \\w \vphantom{ M_{\mathrm{opt},0}(w)}
\end{bmatrix} , \begin{bmatrix}
            v_p \vphantom{M_p(x_p)}  \\z \vphantom{ M_{\mathrm{opt},0}(w)}
\end{bmatrix}  \right) \in V \times V : \begin{bmatrix}
            v_p \vphantom{M_p(x_p)}  \\0 \vphantom{ \stack{0}{\stack{0}{-x_p}}M_{\mathrm{opt},0}(w)}
\end{bmatrix}   = \begin{bmatrix}
    M_p(x_p) \vphantom{ BP_F(\tfrac{1}{\alpha}B^*  \lambda_0)} \\ M_{\mathrm{opt},0}(w)  \vphantom{ \stack{0}{\stack{0}{x_p}}}
\end{bmatrix} - \begin{bmatrix}
      BP_F(\tfrac{1}{\alpha}B^*  \lambda_0+u_\mathrm{ref}) \\ \stack{0}{\stack{0}{-x_p}}
\end{bmatrix}  \right\} .
   \end{equation*}
   Note that
   \begin{equation*}
       \dom{M_0} = \left\{\begin{bmatrix}
            x_p \vphantom{M_p(x_p)}  \\w \vphantom{ M_{\mathrm{opt},0}(w)}
\end{bmatrix} \in  V : 0=M_{\mathrm{opt},x_p}(w)  \right\},
   \end{equation*}
that is, in the domain of this operator are the consistent tuples $(x_p,w)$ for which $w$ corresponds to the optimal solution emanating from initial value $x_p$.

Using this operator, we may rewrite the coupled optimizer-plant DAE \eqref{eq:optimizer_plant_DAE} as
\begin{equation*}
    \frac{\mathrm{d}}{\mathrm{d}t} \begin{pmatrix}
    x_p(t) \\ w(t)
\end{pmatrix} + M_0\begin{pmatrix}
    x_p(t) \\ w(t)
\end{pmatrix} \ni 0.
\end{equation*}
We decompose $M_0=F_1 + F_2$ where $F_1$ is as in \eqref{eq:optimizer_plant_dec} and  $F_2: V \rightrightarrows V$ is defined by
\begin{equation*}
     F_2 = \left\{ \left(\begin{bmatrix}
            x_p \vphantom{M_p(x_p)}  \\w \vphantom{ M_{\mathrm{opt},0}(w)}
\end{bmatrix} , \begin{bmatrix}
            v_p \vphantom{M_p(x_p)}  \\z \vphantom{ M_{\mathrm{opt},0}(w)}
\end{bmatrix}  \right) \in V \times V : \begin{bmatrix}
            v_p \vphantom{M_p(x_p)}  \\0 \vphantom{ \stack{0}{\stack{0}{-x_p}}M_{\mathrm{opt},0}(w)}
\end{bmatrix}   = \begin{bmatrix}
    0 \\ M_{\mathrm{opt},0}(w) -\stack{0}{\stack{0}{-x_p}}
\end{bmatrix}   \right\} ,
\end{equation*}
with $\dom{F_2} = \dom{M_0}$. Consider the two subproblems
\begin{equation*}
    \frac{\mathrm{d}}{\mathrm{d}t} \begin{pmatrix}
    x_p(t) \\ w(t)
\end{pmatrix} + F_i \begin{pmatrix}
    x_p(t) \\ w(t)
\end{pmatrix} \ni 0, \qquad i=1,2,
\end{equation*}
Again, due to Theorem~\ref{thm:homogeneous_Cauchy} there exists a flow $\varphi_t^1: V \to V$ where $ \varphi^1_t(v_0) \coloneqq S_{F_1}(t)v_0$. We denote by $r: V \to V$ the mapping
\begin{equation}\label{eq:defnf}
    r (v) = Q  M_{\mathrm{opt},0}^{-1}(P v),
\end{equation}
where $P: V \to V$, $Q: W \to V$ are given by
\begin{equation*}
    P \stack{x_p}{w} = \stack{0}{\stack{0}{x_{p}}}, \qquad Q w = \stack{0}{w}.
\end{equation*}
The mapping $r$ defined in \eqref{eq:defnf} corresponds to the instantaneous solution of the optimal control problem, that is, it resolves the algebraic constraint. Consequently, we obtain in Algorithm~\ref{alg:Lie_Trotter_MPC} the splitting scheme for the optimizer-plant DAE. 

\begin{algorithm}[htb]
\caption{Lie--Trotter type algorithm.}
\label{alg:Lie_Trotter_MPC}
\begin{algorithmic}[1]
\Statex \textbf{Given}: Initial state $v_0 = v(0)$, sampling time $h>0$.
\For{each $n = 0,1,2,\dots$}
\State Solve $\dot{v}_1 = F_1(v_1)$ with $v_1(t_n) = v_n$, on $[t_n, t_{n+1}]$.
\State Set $v_{n+\frac{1}{2}} = v_1(t_{n+1})$.
\State Set $v_{n+1} = r(v_{n+\frac{1}{2}})$ with $r$ given by \eqref{eq:defnf}.
\EndFor
\end{algorithmic}
\end{algorithm}
\begin{theorem}
Let in system \eqref{eq:control_system} be $f=-M_p$ and $g=B_p$. Fix a sampling time $h > 0$, and let $r$ be given by \eqref{eq:defnf}.  Denote by $\varphi_t^1$ the flow associated with the subproblem governed by $F_1$. For $v_n = (x_n, w_n) \in V$, define the one-step update 
\begin{equation}
    v_{n+1} = r \circ \varphi_h^1(v_n),
\end{equation}
that is, we perform Algorithm~\ref{alg:Lie_Trotter_MPC} with $t_{n+1} = t_n + h$. Then, the sequence $(v_n)_{n \ge 0}$ generated by this recursion coincides with the sequence produced by the MPC approach of Algorithm~\ref{alg:MPC} applied to \eqref{eq:MPC_OCP} with $\tilde{f}=A, \tilde{g}=B$.
\end{theorem}
\begin{proof}
    By definition of the flow $\varphi_h^1$, we have that $v_{n+1/2}$ denotes the solution of the differential equation
    \begin{align}\label{eq:F_1MPC}
        \frac{\mathrm{d}}{\mathrm{d}t} \begin{pmatrix}
    x_p(t) \\ w(t)
\end{pmatrix}  =  \begin{bmatrix}
    -M_p(x_p(t)) +  BP_F(\tfrac{1}{\alpha}B^*  \lambda_0(t)) \\ 0
\end{bmatrix}, \qquad v(t_n)=v_n, \quad t \in [t_n, t_{n+1}]
    \end{align}
    at time $t_{n+1}$. Clearly, the second line of \eqref{eq:F_1} yields that the $W$-component of $\varphi_h^1(v_n)$ coincides with the initial value $w_n=(x_n, (\lambda_n,\lambda_{0n}))$. Consequently, \eqref{eq:F_1} reduces to 
     \begin{align}\label{eq:F_1reducedMPC}
        \frac{\mathrm{d}}{\mathrm{d}t} 
    x_p(t)= -M_p(x_p(t)) +  BP_F(\tfrac{1}{\alpha}B^*  \lambda_{0n}) 
\qquad x_p(t_n)=x_{pn}, \quad t \in [t_n, t_{n+1}].
    \end{align}
    The preceding differential equation has a unique solution $x_p(t)$ on $[t_n,t_{n+1}]$ and the end point evaluation $x_p(t_{n+1})$ coincides with the current state measurement in the Step~4 of Algorithm~\ref{alg:MPC}. Accordingly, we consider the differential-algebraic equation
    \begin{align}\label{eq:F_2MPC}
       \begin{bmatrix}
    I & \\ & 0
\end{bmatrix} \frac{\mathrm{d}}{\mathrm{d}t} \begin{pmatrix}
    x_p(t) \\ w(t)
\end{pmatrix}  =  \begin{bmatrix}
    0 \\  -M_{\mathrm{opt},0}(w(t)) + \stack{0}{\stack{0}{x_p(t)}}
\end{bmatrix}, \qquad v(t_n)=v_{n+1/2}, \quad t \in [t_n, t_{n+1}].
    \end{align}
 Using the same argumentation as above, the first line of \eqref{eq:F_2} yields that the $\R^n$-component of $\varphi_h^2(v_{n+1/2})$ coincides with the initial value $x_{pn}$. Consequently, \eqref{eq:F_2} reduces to 
     \begin{align}\label{eq:F_1_reducedMPC}
        0=  - M_{\mathrm{opt},0}(w(t)) +\stack{0}{\stack{0}{x_{pn}}}
\qquad w(t_n)=w_{n+1/2}, \quad t \in [t_n, t_{n+1}]
    \end{align}
    and the mapping $r$ (instantaneously) assigns to the initial condition $v_{n+1/2}=(x_{p(n+1/2)}, w_{n+1/2})$ the unique solution of the optimal control problem \eqref{eq:oc_with_dynamics}.
\end{proof}
\section{Numerical approximation of optimizer flows -- optimization algorithms}\label{sec:opt}

In the previous section, we introduced the hybrid Lie–Trotter splitting of the coupled optimizer–plant dynamics \eqref{eq:optimizer_plant_dec} and related it to the suboptimal MPC algorithm \ref{alg:subopt_MPC}. In particular, in Theorem~\ref{thm:lt_prox} we have connected a Lie--Trotter approximation with a suboptimal MPC scheme, in which the optimal control problem is solved with a proximal-point method for the optimality system. In this section, we will broaden the scope for the optimal control solver and discuss more general optimization methods; these methods closely correspond to the numerical approximation of the subsystem governed by $F_2$ in \eqref{eq:optimizer_plant_dec} encoding the flow of the optimizer.

That is, in this section we consider different splitting schemes for the bottom right box in Figure~\ref{fig:num_approx} as a more general framework compared to Figure~\ref{fig:subopt}. In contrast to the previous sections, we do not consider the control-reduced optimality system \eqref{eq:red_OS}. Instead, we study the gradient flow associated with the optimality system \eqref{eq:os}, shown in the bottom-right box. In the presence of control constraints, this leads to a set-valued equation, and consequently, the resulting optimizer dynamics are also set-valued; see Figure~\ref{fig:num_approx}. This generalization is essential for capturing modern first-order and operator-splitting methods within a unified dynamical framework and for their inclusion in suboptimal MPC schemes.

\begin{figure}[htb]
\begin{center}
    \scalebox{0.9}{\begin{tikzpicture}[
  >=Latex,
  font=\small,
  block/.style={
    draw,
    thick,
    fill=gray!15,
    rounded corners=1pt,
    inner sep=.1cm,
    align=left,
    text width=6.2cm
  },
  wire/.style={thick, -Latex},
  lab/.style={font=\small, inner sep=1pt}
]

% --- Blocks ---
\node[block] (plant) {
\vspace*{-.3cm}
\begin{align*}
\dot x_p(t) &= -M_p(x_p(t)) + B_p u_p(t)\\
y_p(t) &= x_p(t)
\end{align*}
};

\node[block, right=1.5cm of plant, yshift=-1.5cm] (opt) {
\vspace*{-.2cm}
\begin{align*}
\varepsilon \dot w(t) &\in -\calA(w(t)) - \partial g(w(t)) +\mathcal{B}u_\mathrm{opt}(t) \\
y_\mathrm{opt}(t) &= \mathcal{B}^* w(t)
\end{align*}
};

% --- Coordinates for clean routing ---
\coordinate (plantE) at ($(plant.east)+(0.2,0)$);
\coordinate (optN)   at ($(opt.north)+(0,0.2)$);
\coordinate (optW)   at ($(opt.west)+(-0.2,0)$);
\coordinate (plantS) at ($(plant.south)+(0,-0.2)$);

% --- y_p -> optimizer input (horizontal, then vertical) ---
\draw[wire]
(plant.east) -| node[lab, above]
{$u_\mathrm{opt}(t) = y_p(t)$}
(opt.north);

% --- optimizer output -> plant input (horizontal, then vertical) ---
\draw[wire]
(opt.west) -| node[lab, below]
{$u_p(t)= P_F(-\tfrac{1}{\alpha}B^*  y_\mathrm{opt}(t)+u_\mathrm{ref})$}
(plant.south);

\end{tikzpicture}}
\end{center}
\caption{The bottom right box defines a continuous-time (set-valued) control system. We highlight the splitting into a single-valued part $\calA$ and a subgradient $\partial g$ and apply modern splitting-based numerical algorithms for sums of monotone operators onto the optimizer dynamcis.}
\label{fig:num_approx}
\vspace*{-.2cm}
\end{figure}
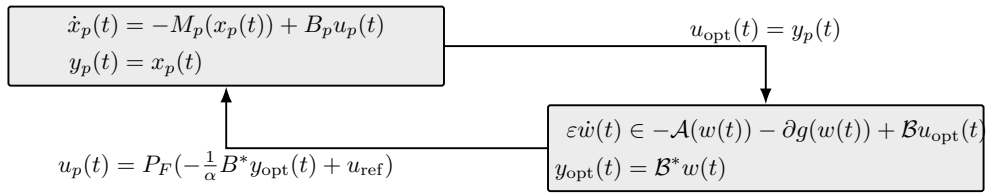

The relation between splitting schemes for differential inclusions and iterative optimization algorithms is well established in monotone operator theory and convex optimization, see \cite[Ch. 26]{BausComb2011}. Nevertheless, this viewpoint is particularly useful in the present setting, since it allows us to interpret a broad class of optimization algorithms as numerical approximations of the optimizer dynamics coupled to the plant. In particular, the specific structure of the optimizer subsystem arising from the optimality system enables a direct interpretation of splitting-based optimization algorithms within the proposed MPC framework.

Due to the optimality condition \eqref{eq:os}, solving the optimal control problem at each MPC step is equivalent to computing a zero of the associated optimality system. A~key observation is that the optimality system \eqref{eq:os} possesses a natural sum structure consisting of a single-valued monotone part and a generally set-valued nonsmooth part. This structure forms the basis for the splitting approaches considered below.

More precisely, for a given initial value $x_0 \in \R^n$, the optimality system \eqref{eq:os} can be written in the form
\begin{equation}\label{eq:os2}
\text{find }  \begin{pmatrix}
         z \\ p
      \end{pmatrix} \text{ such that} \qquad \qquad \qquad 0 \in 
 \calA \begin{pmatrix}
          z \\ p
      \end{pmatrix} +  \calS \begin{pmatrix}
         z \\ p
      \end{pmatrix} + \stack{0}{\stack{0}{x_0}}.
\end{equation}

To emphasize this decomposition and to cover a broader class of structured optimization problems, we adopt the following abstract framework throughout this section:
\begin{enumerate}[label=(\roman*)]
    \item $\calA: W \supset \dom{\calA} \to W$ is a maximally monotone operator representing the single-valued part of the optimality system;
    \item $\calS = \partial g \subset W \times W$ is the subdifferential of a proper, convex, and lower semi-continuous function $g: W \to \R \cup \{+\infty\}$ and models the nonsmooth and constraint-related part;
    \item $b \in W$ denotes a given forcing term.
\end{enumerate}

In the particular case of \eqref{eq:os}, the nonsmooth component arises from the control constraints through the indicator function $g(z,p)=i_F(u)$.

Motivated by \eqref{eq:F_2}, we consider the flow associated with the inclusion \eqref{eq:os2}, i.e., we introduce the dynamical system
\begin{equation}\label{eq:gradient_flow_for_os}
      \varepsilon \frac{\mathrm{d}}{\mathrm{d}t} \begin{pmatrix}
          z(t) \\ p(t)
      \end{pmatrix} \in  -
 \calA \begin{pmatrix}
          z(t) \\ p(t)
      \end{pmatrix} -  \calS \begin{pmatrix}
          z(t) \\ p(t)
      \end{pmatrix} - b
 \qquad\begin{pmatrix}
          z(0) \\ p(0)
      \end{pmatrix} = \begin{pmatrix}
          z_0 \\ p_0
      \end{pmatrix}
\end{equation}
where $\varepsilon > 0$ denotes a time scaling. Note that, to cover \eqref{eq:os}, therein the function $g$ maps $(z,p)=((x,u),p)$ to $g(z,p)=i_F(u)$ so that the control component of $\calS$ contains the subgradient $\partial i_F$ and is zero everywhere else.

We briefly recall some preliminaries from convex analysis before discussing particular algorithms emerging as splitting and discretization schemes of the gradient flow \eqref{eq:gradient_flow_for_os}. Let $\calM: W \rightrightarrows W$ be a possibly set-valued operator and $\gamma > 0$. Define its \textit{resolvent at $\gamma$} by 
\begin{equation*}
    J_{\gamma \calM} \coloneqq (I + \gamma \calM )^{-1}
\end{equation*}
and its \textit{Yosida approximation of index $\gamma$} by
\begin{equation*}
    \calM_\gamma  \coloneqq \tfrac{1}{\gamma}(I-J_{\gamma \calM}).
\end{equation*}
In the case $\calM= \partial g$, i.e., if the $\calM$ is a subgradient, its resolvent at $\gamma > 0$ is also denoted by the $\prox$-operator
\begin{equation*}
    \prox_{\gamma g} \coloneqq J_{\gamma \partial g}.
\end{equation*}
\begin{proposition}[{\cite[Cor. 23.11]{BausComb2011}}]
    Let $\calM:W \rightrightarrows W$ be maximally monotone and $\gamma > 0$. Then, the following assertions hold:
    \begin{enumerate}[label=(\roman*)]
        \item the resolvent $J_{\gamma \calM}: W \to W$ and $I-J_{\gamma \calM}: W \to W$ are firmly nonexpansive and maximally monotone;
        \item $\calM_\gamma $ is maximally monotone;
        \item $\calM_\gamma $ is firmly nonexpansive;
        \item $ \zer \calM = \zer{\calM_\gamma}=\operatorname{Fix} (J_{\gamma \calM}) \eqqcolon  \left\lbrace w \in W : J_{\gamma \calM}(w)=w \right\rbrace $.
    \end{enumerate}
\end{proposition}
\begin{proof}
The first three claims (i)--(iii) were proven in \cite[Cor. 23.11]{BausComb2011}. The claim (iv) follows from \cite[Prop. 23.38]{BausComb2011}.
\end{proof}

\subsection{Proximal-gradient and forward-backward algorithm}

Let $\varphi_t: W \to W$ denote the flow associated with the differential inclusion \eqref{eq:gradient_flow_for_os}, precisely $\varphi_t(w_0)=S_{\calM}(t)w_0$. Accordingly, let $\varphi_t^i: W \to W$ be the flows associated with the subproblems 
\begin{equation}
     \varepsilon \frac{\mathrm{d}}{\mathrm{d}t} 
        w_1(t) = -
 \calA  w_1(t)  -  b , \qquad   \varepsilon \frac{\mathrm{d}}{\mathrm{d}t} 
        w_2(t) \in -
 \calS  w_2(t) ,
\end{equation}
respectively. Now, the Lie--Trotter scheme approximates the original flow by the composition
\begin{equation*}
    \varphi_h \approx \varphi_h^2 \circ \varphi_h^1.
\end{equation*}
Recall the numerical approximations of the subflows $\varphi_t^i$ given in \eqref{eq:expl_euler}, \eqref{eq:impl_euler}. An explicit Euler step in the first subproblem and an implicit Euler step in the second subproblem is rewritten by the update rule
\begin{equation}\label{eq:phi_h_star}
    \phi_h^{3*} = \psi_h^{2i} \circ \psi_h^{1e}.
\end{equation}
For further combinations of first-order schemes, we refer to \cite{blanes2024splitting}. If we treat $\calM$ as the sum of the two monotone operator $\calA+b$ and $\calS$, then we may apply the \textit{forward-backward algorithm}
\begin{equation*}
     w_{n+1}=J_{\gamma \calS} (w_n- \gamma \calA(w_n) - \gamma b), \qquad  n \in \N
\end{equation*}
see also \cite[Thm. 26.14]{BausComb2011}. The subsequent theorem provides the relation between time discretization schemes for the subflows $\varphi_t^i$ and numerical algorithms for the zeros of sums of maximally monotone operators.
\begin{theorem}
Set step-size $h > 0$ and a time-scaling $\varepsilon > 0$. Let $g$ be a proper convex l.s.c.\ function and $\calS = \partial g$. Then, 
\begin{equation*}
    \phi_h^{3*} (w) = J_{h/\varepsilon \calS}(w-h/\varepsilon \calA(w) - h/\varepsilon b) = \prox_{h/\varepsilon g} (w- h/\varepsilon \calA(w) - h/\varepsilon b)
\end{equation*}
for all $w \in W$. Hence, for an initial value $w_0 \in W$, the sequences $(w_n)_{n \in \N}$ produced by the update rules
  \begin{enumerate}[label=(\roman*)]
      \item Explicit/-implicit Euler for optimizer system: $w_{n+1}=\phi_h^{3*}(w_n)$;
      \item Forward-backward for $(\calA+b)+\calS$: $w_{n+1}=J_{h/\varepsilon \calS}(w_n-h/\varepsilon \calA(w_n) - h/\varepsilon b)$;
      \item Proximal gradient: $w_{n+1}=\prox_{h/\varepsilon g} (w_n- h/\varepsilon \calA(w_n) - h/\varepsilon b)$,
  \end{enumerate}
    coincide. 
\end{theorem}
\begin{proof}
By definition, we obtain for each $n \in \N$:
    \begin{align*}
        \phi_h^{3*}(w_n) &= \psi_h^{2i} \circ \psi_h^{1e} (w_n) \\
        &= \psi_h^{2i} (w_n- h/\varepsilon \calA(w_n) - h/\varepsilon b) \\
        &= (I+h/\varepsilon \calS)^{-1 } (w_n- h/\varepsilon \calA(w_n) - h/\varepsilon b) \\
        &= J_{h/\varepsilon \calS} (w_n- h/\varepsilon \calA(w_n) - h/\varepsilon b) \\
        &= \prox_{h/\varepsilon g}(w_n- h/\varepsilon \calA(w_n) - h/\varepsilon b), \\
        %&= F_{h/\varepsilon}(w_n),
    \end{align*}
    which proves the claim.
\end{proof}
\subsection{Peaceman-Rachford algorithm}
Let $\calM: W \rightrightarrows W$ be a possibly set-valued operator and $\gamma > 0$. Define its \textit{reflected resolvent at $\gamma$} by 
\begin{equation*}
    R_{\gamma \calM} \coloneqq 2J_{\gamma \calM} - I.
\end{equation*}
The \textit{Peaceman-Rachford algorithm}~\cite{peaceman1955numerical} for the sum of the two monotone operators $\calS$ and $\calA + b$ is then given by the update rule
\begin{equation*}
    w_{n+1} =( R_{\gamma \calS} \circ R_{\gamma (\calA +  b)})(w_n), \qquad n \in \N,
\end{equation*}
see also \cite[Prop. 26.13]{BausComb2011}. Convergence results for this algorithm in the case of set-valued operators $\calA$ and $\calS$ were given in \cite{lions1979splitting}. The subsequent theorem gives the correspondence between the Peaceman-Rachford algorithm for the sum $\calS + \calA +b$ and the Lie--Trotter splitting where each subflow is numerically approximated via the implicit midpoint rule.
\begin{theorem}
Set step-size $h > 0$ and a time-scaling $\varepsilon > 0$. Let $g$ be a proper convex l.s.c.\ function and $\calS = \partial g$. Then, 
\begin{equation*}
    (\psi_h^{2e} \circ \psi_h^{2i} \circ \psi_h^{1e} \circ \psi_h^{1i}) (w) = (R_{\gamma \calS} \circ R_{\gamma (\calA +  b)})(w)
\end{equation*}
for all $w \in W$. Hence, for an initial value $w_0 \in W$, the sequences $(w_n)_{n \in \N}$ produced by the update rules
  \begin{enumerate}[label=(\roman*)]
      \item $w_{n+1}=(\psi_h^{2e} \circ \psi_h^{2i} \circ \psi_h^{1e} \circ \psi_h^{1i})(w_n) $;
      \item $w_{n+1}=(R_{\gamma \calS} \circ R_{\gamma (\calA +  b)})(w_n)$,
  \end{enumerate}
  coincide.
\end{theorem}
\begin{proof}
    Observe that for all $w \in W$: 
    \begin{align*}
         (\psi_h^{1e} \circ \psi_h^{1i})(w)
         &= \psi_h^{1e} (J_{h/\varepsilon(\calA + b)}(w)) \\
         &= (I-h/\varepsilon(\calA +b))(J_{h/\varepsilon(\calA + b)}(w)) \\
         &= (I+h/\varepsilon(\calA +b)-2h/\varepsilon(\calA +b))(J_{h/\varepsilon(\calA + b)}(w)) \\
         &= I -2(I-J_{h/\varepsilon(\calA+b)}) \\
         & =R_{h/\varepsilon(\calA+b)}
    \end{align*}
    and replacing $(\calA+b)$ by $\calS$ in the above computation yields 
     \begin{align*}
         (\psi_h^{2e} \circ \psi_h^{2i})(w)
         & =R_{h/\varepsilon \calS}(w).
    \end{align*}
    This proves the assertion.
\end{proof}

\section{Conclusion}\label{sec13}

We introduced a framework for suboptimal MPC based on the interconnection of the plant and the  optimizer dynamics, formulated via (quasi-)monotone operators. This viewpoint yields a well-posed continuous-time closed-loop system and provides a unified interpretation of suboptimal MPC schemes. In particular, we showed that such algorithms can be understood as Lie--Trotter splitting methods applied to the coupled dynamics. This connection links real-time optimization with operator splitting and offers a foundation for the systematic design of efficient MPC schemes.

\bibliographystyle{abbrv}

\bibliography{sn-bibliography}

\newpage

\end{document}